\documentclass[ejs]{imsart}

\pdfoutput=1
\RequirePackage{amsthm,amsmath,amsfonts,amssymb}
\RequirePackage[authoryear]{natbib}
\RequirePackage[colorlinks,citecolor=blue,urlcolor=blue]{hyperref}
\RequirePackage{graphicx}

\usepackage[utf8]{inputenc} 
\usepackage[T1]{fontenc}    
\usepackage{hyperref}       
\usepackage{url}            
\usepackage{booktabs}       
\usepackage{amsfonts}       
\usepackage{amsmath}
\usepackage{amsthm}
\usepackage{amssymb}
\usepackage{microtype}      
\usepackage{xcolor}         
\usepackage{caption}
\usepackage{comment}
\usepackage{mathrsfs}
\usepackage{graphicx}
\usepackage{cleveref}

\startlocaldefs
\theoremstyle{plain}

\newtheorem{theorem}{Theorem}[section]
\newtheorem{lemma}[theorem]{Lemma}

\newtheorem{proposition}[theorem]{Proposition}

\theoremstyle{remark}

\allowdisplaybreaks

\newcommand{\pn}{p^n}
\newcommand{\qn}{q^n}

\newcommand{\htheta}{\wh{\theta}}

\newcommand{\btheta}{\overline{\theta}}
\newcommand{\thetastar}{\theta^\star}
\newcommand{\tloss}{\wt{\loss}}

\newcommand{\shloss}{\loss^{(\eta)}}
\newcommand{\shLoss}{\Loss^{(\eta)}}

\newcommand{\dd}{\mathrm{d}}

  \newcommand{\Y}{\mathcal{Y}}
\newcommand{\F}{\mathcal{F}}
\newcommand{\real}{\mathbb{R}}

\newcommand{\Sw}{\mathcal{S}}

\newcommand{\II}[1]{\mathbb{I}_{\left\{#1\right\}}}
\newcommand{\PP}[1]{\mathbb{P}\left[#1\right]}

\newcommand{\EE}[1]{\mathbb{E}\left[#1\right]}

\newcommand{\EEcc}[2]{\mathbb{E}\left[\left.#1\right|#2\right]}

\def\argmin{\mathop{\mbox{ arg\,min}}}

\newcommand{\ra}{\rightarrow}

\newcommand{\siprod}[2]{\langle#1,#2\rangle}
\newcommand{\iprod}[2]{\left\langle#1,#2\right\rangle}

\newcommand{\norm}[1]{\left\|#1\right\|}
\newcommand{\bnorm}[1]{\bigl\|#1\bigr\|}

\newcommand{\twonorm}[1]{\norm{#1}_2}

\newcommand{\ev}[1]{\left\{#1\right\}}

\newcommand{\pa}[1]{\left(#1\right)}
\newcommand{\bpa}[1]{\bigl(#1\bigr)}

\newcommand{\wh}{\widehat}
\newcommand{\wt}{\widetilde}

\newcommand{\transpose}{^\mathsf{\scriptscriptstyle T}}

\newcommand{\loss}{\ell}
\newcommand{\Loss}{\mathcal{L}}

\usepackage{todonotes}
\definecolor{PalePurp}{rgb}{0.66,0.57,0.66}

\newcommand{\regret}{\mathrm{regret}}

\usepackage[normalem]{ulem}
\usepackage{enumitem}

\newcommand{\Bb}{\mathcal B}

\newcommand{\E}{\mathbb E}

\newcommand{\Ell}{\mathcal L}

\newcommand{\Id}{\mathrm{Id}}

\newcommand{\inner}[2]{\langle #1, #2 \rangle}

\newcommand{\Pp}{\mathbb P}

\newcommand{\regretp}{\regret_{p^n}}
\newcommand{\regretq}{\regret_{q^n}}
\newcommand{\regretql}[1]{\regret_{q^n, {#1}}}

\newcommand{\supp}{\mathrm{supp}}

\definecolor{zaffre}{RGB}{0, 20, 168}

\endlocaldefs

\begin{document}

\begin{frontmatter}
\title{Confidence Sequences for Generalized Linear Models via Regret Analysis}
\runtitle{Confidence Sequences for GLMs via Regret Analysis}

\begin{aug}
\author[A]{\fnms{Eugenio}~\snm{Clerico} \ead[label=e1]{eugenio.clerico@gmail.com}},
\author[A]{\fnms{Hamish}~\snm{Flynn}\ead[label=e2]{hamishedward.flynn@upf.edu}},
\author[B]{\fnms{Wojciech}~\snm{Kot\l owski}\ead[label=e3]{kotlow@gmail.com}},
\and
\author[A]{\fnms{Gergely}~\snm{Neu}\ead[label=e4]{gergely.neu@gmail.com}}
\address[A]{Universitat Pompeu Fabra, Barcelona, Spain.\printead[presep={\\}]{e1,e2,e4}}
\address[B]{Pozna\'n University of Technology, Pozna\'n, Poland.\printead[presep={\\}]{e3}}

\end{aug}

\begin{abstract}
We develop a methodology for constructing confidence sets for parameters of statistical models via a reduction to 
sequential prediction. Our key observation is that for any generalized linear model (GLM), one can construct an 
associated game of sequential probability assignment such that achieving low regret in the game implies a 
high-probability upper bound on the excess likelihood of the true parameter of the GLM. This allows us to develop a 
scheme that we call \emph{online-to-confidence-set conversions}, which effectively reduces the problem of proving the 
desired statistical claim to an algorithmic question. We study two varieties of this conversion 
scheme: 1) \emph{analytical} conversions that only require proving the existence of algorithms with low regret and 
provide confidence sets centered at the maximum-likelihood estimator 2) \emph{algorithmic} conversions that actively 
leverage the output of the online algorithm to construct confidence sets (and may be centered at other, adaptively 
constructed point estimators). The resulting methodology recovers all state-of-the-art confidence set constructions 
within a single framework, and also provides several new types of confidence sets that were previously unknown in the 
literature.
\end{abstract}



\end{frontmatter}

\section{Introduction}\label{sec:intro}

Building confidence sets for parameters of statistical models is one of the most fundamental questions of statistics. 
In this paper, we consider this problem in the context of generalized linear models (GLMs), where one has access to a 
data set of observations $(X^n,Y^n) = (X_t,Y_t)_{t=1}^n$. Here, $X_t \in \real^d$ is a vector of \emph{covariates} 
(or \emph{features}), $Y_t \in \real$ is a real-valued \emph{label}, and the likelihood of the labels $Y_t$ 
is given by the exponential-family model
\[
 p(y|X_t,\thetastar) = \exp\bpa{\iprod{\theta^\star}{X_t}y - \psi(\iprod{\theta^\star}{X_t})} h(y)\,,
\]
with $\thetastar \in \real^d$ the unknown parameter, $\psi:\real\ra\real$ a convex function (often called the 
\emph{log-partition function}), and $h:\real\ra\real$ the \emph{reference distribution} (independent of $X_t$ or 
$\theta^\star$).
The model can be alternatively written as $Y_t = \mu(\iprod{\theta^\star}{X_t}) + 
\varepsilon_t$, where $\mu = \psi'$ and $\varepsilon_t$ is zero-mean noise whose precise distribution depends on 
$\iprod{\theta^\star}{X_t}$. 
We allow the covariates to be selected sequentially, which means that the distribution of $X_{t+1}$ is determined by the 
sequence $(X^t,Y^t)$. We will use $\mathcal F_t = \sigma(X_1, Y_1, X_2, 
Y_2, \dots, X_t, Y_t, X_{t+1})$ to denote the $\sigma$-field generated by the sequence of observations up to index $t$ and 
the covariate $X_{t+1}$. This model captures many problems of 
fundamental interest: the choice $\psi: z \mapsto \frac {z^2}{2}$ yields linear models with Gaussian 
noise, while logistic regression can be recovered by considering Bernoulli labels $Y_t = 
\mathrm{Bernoulli}\bpa{\mu(\iprod{\theta^\star}{X_t})}$, with $\mu: z\mapsto \frac{1}{1 + e^{-z}}$ being the sigmoid 
link function. The corresponding log-partition function in 
this case is the logistic loss $\psi(z) = \log\pa{1 + e^z}$.

In this work, we aim to develop tight \emph{confidence sequences} for the true parameter $\thetastar$, which are 
sequences of sets $\Theta_1,\dots,\Theta_n\subseteq \real^d$ such that each $\Theta_t$ is determined by $(X^t, Y^t)$ and 
$\thetastar$ is included in \emph{all} sets with high probability.
Formally, we require $\PP{\exists n: \thetastar\not\in\Theta_n} \le \delta$ for some given $\delta > 0$. A set 
satisfying $\PP{\thetastar\not\in\Theta_n} \le \delta$ for a specific value of $n$ will be referred to as a 
\emph{confidence set}. We will 
largely focus on confidence sets and sequences in which each set is defined as the set of parameters whose likelihood 
is nearly maximal:
\begin{equation}\label{eq:likelihood_conf_set}
 \Theta_n = \ev{\theta\in\real^d: \frac{\prod_{t=1}^n p(Y_t|X_t,\theta)}{\sup_{\theta'\in\real^d} \prod_{t=1}^n 
p(Y_t|X_t,\theta')} \ge e^{-\beta_n}}.
\end{equation}
The challenge then is to design suitable choices of the \emph{confidence width} $\beta_n > 0$ such that 
$\Theta_1, \Theta_2, \dots$ is guaranteed to be a valid confidence sequence. In this paper, we develop a methodology that reduces 
the \emph{statistical} problem of designing an appropriate sequence of confidence widths to the \emph{algorithmic} 
problem of showing the existence of a sequential prediction algorithm with guaranteed bounds on their performance (as 
measured by the standard notion of \emph{regret}, to be defined shortly).

Our methodology can be most easily introduced via the following simple example that considers the linear-Gaussian case. 
For such models, the confidence set of Equation~\eqref{eq:likelihood_conf_set} can be equivalently written in terms of 
the squared loss function $\ell_t(\theta) = \frac 12 \pa{\iprod{\theta}{X_t} - Y_t}^2$, as 
\begin{equation}\label{eq:linear_conf_set}
 \Theta_n = \ev{\theta\in\real^d: \inf_{\theta'\in\real^d} \sum_{t=1}^n \pa{{\ell}_t(\theta) - 
{\ell}_t(\theta')} \le \beta_n}.
\end{equation}
For the sake of simplicity, let us now suppose that the maximum-likelihood estimator (MLE) 
$\htheta_n = \argmin_{\theta\in\real^d} 
\sum_{t=1}^n \ell_t(\theta)$ exists. In this case, it can be easily shown that the set $\Theta_n$ is an ellipsoid centered at 
$\htheta_n$.
In order to provide a suitable choice of the width $\beta_n$ of the ellipsoid, we will consider a game between an 
\emph{online learner} and its \emph{environment}, where in each round $t=1,2\dots,n$, the following steps are repeated:
\begin{enumerate}
 \item the environment reveals $X_t$ to the online learner,
 \item the online learner picks a parameter $\theta_t \in \real^d$,
 \item the environment reveals $Y_t$ to the online learner, and
 \item the online learner incurs the loss $\ell_t(\theta_t)$.
\end{enumerate}
The performance of the online learner is measured in terms of its \emph{regret} against a (potentially 
data-dependent) comparator $\theta'\in\real^d$, defined as
\[
 \regret_n(\theta') = \sum_{t=1}^n \bpa{\ell_t(\theta_t) - \ell_t(\theta')}.
\]
Our key observation is that the excess log-likelihood of the true parameter $\thetastar$ relative to the MLE can be decomposed 
as follows:
\begin{equation*}
 \sum_{t=1}^n \pa{\ell_t\pa{\thetastar} - \ell_t\bpa{\htheta_n}} = \underbrace{\sum_{t=1}^n \pa{\ell_t(\theta_t) - \ell_t\bpa{\htheta_n}}}_{\regret_n\bpa{\htheta_n}} + 
\underbrace{\sum_{t=1}^n \pa{\ell_t(\thetastar) - \ell_t(\theta_t)}}_{\text{$\le \log \frac 1\delta$ with high prob}}.
\end{equation*}
Here, we have noticed that the first term in the decomposition corresponds to the regret of the online 
learner against $\htheta_n$, and that since the online learner has to make its prediction $\theta_t$ before seeing the 
label $Y_t$, it cannot have lower squared-loss prediction error than the true parameter $\thetastar$, in expectation at least. As a result, the 
last term in the decomposition can be shown to be a light-tailed supermartingale. Notably, the argument holds for 
\emph{any} online learning algorithm, and thus in order to construct a confidence set of the form given in 
Equation~\eqref{eq:linear_conf_set} with small $\beta_n$, one only needs to show the \emph{existence} of an online 
learner with low regret against the data-dependent comparator point $\htheta_n$. 

For the online-learning game considered above, the best possible regret guarantees are of the order $\max_{t} Y_t^2 
\cdot d\log n$, achieved by the so-called Vovk--Azoury--Warmuth forecaster (cf.~\citealp{AW01}, \citealp{Vov01}, and 
Section~11.8 in \citealp{cesa2006prediction}). Ultimately, this leads to a confidence width $\beta_n$ of the order $d 
\pa{\log 
n}^2 + \log \frac 1\delta$, which is inferior to the optimal rate of $d\log n + \log \frac 1\delta$,
suggesting that the reduction we propose above is unlikely to yield tight confidence sets, at least in the simple form 
defined above. We address this issue by 
considering a closely related online learning game of \emph{sequential probability assignment}, which allows us to 
derive confidence sequences of \emph{optimal} width in a number of important cases of interest. 

Besides making the above argument fully rigorous, in what follows we will extend this simple approach along multiple 
axes, such as going beyond linear models, providing reductions to more flexible and powerful online algorithms, or 
considering confidence sets of alternative shapes. We will refer to our general methodology as 
\emph{Online-to-Confidence-Set Conversions} (following the terminology introduced by \cite{abbasi2012online}, 
as discussed below).
One aspect of the conversion scheme that we will particularly develop is the manner in which the online 
algorithm is used for constructing the confidence sequence. In the argument above, we have used the online algorithm 
only for the sake of analysis: there was no need to ever run the algorithm, as all that the argument needed was showing 
the existence of a method with low regret. We will call such conversions \emph{analytic}. Alternatively, one can 
use the output of the online algorithm more actively in constructing the confidence sequence, by centering the sets at 
parameters other than the MLE $\htheta_n$. We will refer to such conversion schemes as \emph{algorithmic}.

The idea of using sequential predictions to construct confidence sets is far from being new: it goes back at least 
to the work of \citet{RS70}. More recently, the same idea has resurfaced in 
the work of \cite{abbasi2012online} in a form very similar to our framework. We acknowledge this similarity by 
adopting their aptly chosen term ``online-to-confidence-set conversion'', whose meaning we expand significantly in our 
work. In particular, the results of \cite{abbasi2012online} fall into the category of \emph{algorithmic} 
online-to-confidence-set conversions; our framework highlights the broader context of their work by placing it into a 
larger system of reduction schemes. More recently, online-to-confidence-set conversions were used in a variety of 
settings, but mostly in an ad-hoc way specialized to narrow subclasses of GLMs 
\citep{jun2017scalable,lee2024improved,emmenegger2023likelihood}. Further related works that make use 
of similar ideas without making the connection with regret analysis explicit 
include \citet{gales2022norm}, \citet{flynn2023improved}, and \citet{lee2024unified}. The results in the present work recover all of these 
as special cases, and in several cases improve over them in various aspects.

More broadly, this work fits into a recent wave of statistical literature sometimes referred to as 
\emph{algorithmic statistics}. 
Algorithmic statistics is intended here\footnote{While closely related, the work of \citet{gacs2001algorithmic} titled 
``Algorithmic Statistics'' studies a more narrow set of statistical approaches, which we regard as only one part 
algorithmic statistics in our terminology.} as the statistical counterpart of \emph{algorithmic probability theory}, an 
alternative foundation to classic measure-theoretic probability theory based on game-theoretic and algorithmic 
constructions rooted in the later works of Kolmogorov---see the book of \citet{SV01} for a detailed development of this 
theory, and \citet{vovk2003kolmogorov} for a brief historical overview.
The most notable representatives of the line of work we refer to as algorithmic statistics are the works 
\cite{orabona2024tight, 
waudby2023estimating}, which reduce the problem of mean estimation of bounded random variables to a sequential 
prediction problem, constructing a martingale whose concentration properties are exploited. In particular, 
\cite{waudby2023estimating} takes an approach where by explicitly running a sequential online strategy one builds the 
martingale, while the take of \cite{orabona2024tight} is closer to our analytic method, where the sequential strategy 
does not need to be played in practice as the desired result directly follows from the existence of a suitable regret 
bound for the strategy considered. A similar analytic reduction was also proposed in the online-to-PAC framework of 
\cite{lugosi2024online}, where the existence of a regret bound for a specifically designed online game yields a  
generalization bound for a statistical learning problem. All of these works drew inspiration (at least implicitly) from 
\citet{RS17}, who were the first to explicitly point out the tight connection between concentration inequalities and 
regret analysis, and have proved a general \emph{equivalence} result between martingale tail bounds 
and regret bounds. In a way, all works listed above are merely turning this general observation into concrete, 
quantitative results by repurposing regret minimization algorithms for statistical inference, 
often leading to major improvements over the best previously known results in this context. Our work fits directly into 
this theme as well.

An important additional note about related work is that the results presented in the present manuscript have a 
considerable overlap with the very recent preprint of \citet{kirschner2025confidence} that proposes a general framework 
for deriving confidence sequences. The basic principles underlying their framework are essentially the same as the ones 
we have used as starting point, and many general observations appear in both works. Most notably, our analytic 
online-to-confidence-set conversion appears in nearly identical form in their Section~3.3. Unlike our 
work that exhibits a broad range of confidence sets from this specific reduction, \citet{kirschner2025confidence} 
operate at a different level of generality and present online-to-confidence-set conversions as one of several 
methodologies for developing confidence sets via sequential likelihood ratios. We regard this parallel\footnote{We have 
verified via personal communication with the authors that they agree with this interpretation.} work as complimentary to 
ours.

The rest of the paper is organized as follows. In Section~\ref{sec:glm_model}, we fill the gaps in the definitions 
given above by introducing the full set of our assumptions and describing our online-to-confidence-set framework in 
full technical detail. In Section~\ref{sec:regret}, we present a variety of regret bounds that we will combine with 
our reduction framework in Section~\ref{sec:cases} to derive concrete confidence sets for GLMs. We close with a 
final discussion of the framework, the results, and future research directions in Section~\ref{sec:conclusion}.

\textbf{Notation.} 
For any measurable space $\Sw$, we use $\Delta_{\Sw}$ to denote the set of all probability distributions on $\Sw$, and 
we will use the notation $s^n=\pa{s_1,s_2,\dots,s_n}$ to denote sequences in $\Sw^n$. We will use $\mathbb{I}$ to 
denote the 0-1 indicator function, taking value $\II{E}=1$ if the logical expression $E$ is true and $\II{E} = 0$ 
otherwise. For any positive integer $d$, $[d] = \{1, \dots, d\}$ is the set of integers from 1 to $d$.

\section{Online-to-confidence-set conversions}\label{sec:glm_model}
We now introduce our general framework for constructing confidence sets via a reduction to regret analysis of 
online algorithms. We will consider generalized linear models (GLMs) as described at the beginning of 
Section~\ref{sec:intro}, and will use the notation
$$\ell_t(\theta) = -\inner{\theta}{X_t}Y_t + \psi(\inner{\theta}{X_t}) -\log h(Y_t)\,,$$ 
corresponding to the negative log likelihood $-\log(p(Y_t|X_t, \theta))$ of $Y_t$ given $X_t$ and $\theta$. In our 
framework, we will often think of $\ell_t$ as a \emph{loss function} associated with the statistical model, 
and study algorithms that aim to ``minimize'' this loss in an appropriate sense. Notice that 
the last term in $\ell_t$ is independent of $\theta$, and can be often omitted when analyzing such algorithms.

Instead of the online-learning game outlined in Section~\ref{sec:intro}, we will consider the more flexible setup of 
\emph{sequential probability assignment} (also often called \emph{online prediction under the logarithmic loss}). 
A crucial component of this setup is the \emph{logarithmic loss} (or simply the \emph{log loss}) associated with a 
distribution $P\in\Delta_{\mathcal Y}$ with density $p$ and example $X_t,Y_t$, defined as
\[
\mathcal{L}_t(p) = -\log p(Y_t).
\]
Concretely, we we will consider a sequential game between an online learner and its environment, with the following 
steps repeated in each round $t=1,2,\dots,n$:
\begin{enumerate}
 \item the environment reveals $X_t$ to the online learner,
 \item the online learner picks a distribution $P_t \in \Delta_{\mathcal{Y}}$ with density function $p_t$,
 \item the environment reveals $Y_t$ to the online learner, and
 \item the online learner incurs the log loss $\Loss_t(p_t) = - \log p_t(Y_t)$.
\end{enumerate}
The performance of an online learning algorithm producing outputs $\pn = \pa{p_1, \dots, p_n}$ is measured in terms 
of its \emph{regret} against a comparator strategy that makes its predictions according to $p(\cdot|X_t,\bar{\theta})$ 
for some fixed $\bar{\theta} \in \Theta$:
\begin{equation}\label{eq:regretbayes}
\regretp(\bar{\theta}) = \sum_{t=1}^{n} \pa{\mathcal{L}_{t}(p_t) - \Loss_t\pa{p(\cdot|X_t,\bar{\theta})}} = 
\sum_{t=1}^{n} \pa{\mathcal{L}_{t}(p_t) - \ell_t(\bar{\theta})}\,,
\end{equation}
where we noticed in the last step that $\Loss_t\pa{p(\cdot|X_t,\bar{\theta})} = \ell_t(\bar{\theta})$.

We will play special attention to \emph{mixture forecasters} that pick a distribution $q_t \in \Delta_{\Theta}$ over 
the parameter space in each round $t$, and use the following mixture distribution for predicting $Y_t$:
\[
 p_t = \int_{\Theta} p(\cdot|X_t,\theta) \dd q_t(\theta).
\]
With a slight abuse of our earlier notation, we use $\Loss_t$ to denote loss associated with a mixture 
$q\in\Delta_{\Theta}$: 
\begin{equation}\label{eq:mixloss}
 \Loss_t(q) = -\log \int_{\Theta} p(Y_t|X_t,\theta) \dd q(\theta) = - \log \int_{\Theta} e^{-\loss_t(\theta)} \dd 
q(\theta).
\end{equation}
The regret of a mixture forecaster producing the sequence $\qn = \pa{q_1,\dots,q_n}$ is defined accordingly as
\begin{equation}\label{eq:regretq}
\regretq(\bar{\theta}) = \sum_{t=1}^{n} \pa{\mathcal{L}_{t}(q_t) - \ell_t(\bar{\theta})}\,.
\end{equation}
If $q^n = \pa{q_1, \dots, q_n}$ is a sequence of Dirac delta distributions supported on $\theta^n = 
\pa{\theta_1,\dots,\theta_n}$, the game corresponds to the setup described in Section~\ref{sec:intro}, and we write
\begin{equation}\label{eq:regret_det}
\regret_{\theta^n}(\bar{\theta}) = \sum_{t=1}^{n} \pa{\ell_t(\theta_t) - \ell_t(\bar{\theta})}.
\end{equation}
The problem of designing algorithms with guaranteed bounds on their worst-case regret is well-studied within 
theoretical computer science, statistics and machine-learning theory. We review the most effective methods and state 
the guarantees most relevant to our setting in Section~\ref{sec:regret}.

In what follows, we propose a comprehensive methodology for deriving confidence sets for GLMs using various reductions 
to the sequential prediction game defined above. We will collectively refer to these techniques as 
\emph{online-to-confidence-set conversions}, and classify them into the following two types: 
\emph{analytic} conversions that prove validity of confidence sets centred at arbitrary data-dependent parameters, and 
only use regret analysis within the proofs, and \emph{algorithmic} conversions that make use of the sequence of 
predictions made by the online algorithm for constructing the confidence set. Both conversion schemes build on the 
concept of \emph{sequential likelihood ratios}, which we present first below, followed by the detailed 
description of the conversions.

\subsection{Sequential likelihood ratios}\label{sec:likelihood-ratio-method}
For any $\F_{t-1}$-measurable prediction $p_t\in\Delta_{\mathcal{Y}}$, the difference 
$\bpa{\ell_t(\theta^\star) - \mathcal{L}_t(p_t)}$ can be seen to be the logarithm of a likelihood ratio statistic: 
\begin{equation*}
\begin{split}
\ell_t(\thetastar) - \mathcal{L}_t(p_t) &= \log p_t(Y_t) - \log(p(Y_t|X_t, \theta^{\star})) = 
\log\left(\frac{p_t(Y_t)}{p(Y_t|X_t, \thetastar)}\right).
\end{split}
\end{equation*}
Thus, the sum of these differences is a \emph{sequential likelihood ratio}, to which the following classic result 
applies:
\begin{proposition}\label{prop:generic_lbd}
Let $\pn = (p_1,\dots,p_n)$ be a sequence of distributions over $\mathcal{Y}$ such that each $p_t$ is 
$\F_{t-1}$-measurable. Then, for any $\delta > 0$,
\begin{equation}\label{eq:no-hypercompression}
\PP{\exists n \geq 1: \sum_{t=1}^{n}\ell_t(\theta^\star) - \sum_{t=1}^{n}\mathcal{L}_t(p_t) \geq 
\log(1/\delta)} \leq \delta\,.
\end{equation}
\end{proposition}
\begin{proof}
First, note that we can write
\begin{equation*}
\begin{split}
\sum_{t=1}^{n}\bpa{\ell_t(\thetastar) - \mathcal{L}_t(p_t)} &= \sum_{t=1}^{n}\log p_t(Y_t) - 
\sum_{t=1}^{n}\log(p(Y_t|X_t, \theta^{\star}))\\
&= \log\left(\frac{\prod_{t=1}^{n}p_t(Y_t)}{\prod_{t=1}^{n}p(Y_t|X_t, \thetastar)}\right)\,.
\end{split}
\end{equation*}
Let $M_n = \prod_{t=1}^n \frac{p_t(Y_t)}{p(Y_t|X_t,\thetastar)}$. Using that $p_t$ and $X_t$ are $\mathcal F_{t-1}$ 
measurable, we easily see that this defines a non-negative martingale. Indeed, $$\E[M_t|\F_{t-1}] = M_{t-1}\int_\mathcal Y 
\frac{p_t(y)p(y|X_t,\thetastar)}{p(y|X_t,\thetastar)}\dd y = M_{t-1}\int_\mathcal Y p_t(y)\dd y = 
M_{t-1}\,.$$ Applying Ville's inequality yields
$$\mathbb{P}\left[\forall n \geq 1: \log M_n \le\log\frac{1}{\delta}\right] \geq 1 - \delta\,.$$
\end{proof}
This result is of course classic: it can be traced back to \citet{Vil39}, \citet{Wal45}, and \citet{Rob70}, and is sometimes 
known as the ``no-hypercompression inequality'' (cf.~\cite{grunwald2007minimum}, Chapter~3).
In words, it states that no matter how each prediction $p_t$ is selected, the associated total log 
loss will typically be larger than the loss incurred by the ``oracle'' predictor that knows the true parameter 
$\thetastar$. This observation can already be used to construct valid confidence sequences for $\thetastar$: indeed, 
one just needs to execute an online learning algorithm to produce a sequence of 
predictions $p_1,\dots,p_n$, and consider the set of all parameters $\theta$ with log-likelihood not much smaller than 
$\sum_{t=1}^n \Loss_t(p_t)$. This is in fact a classic recipe that has been used for designing confidence sets at 
least since the work of \citet{RS70}. The tightness of the resulting confidence sets, however, depends on the quality 
of the sequence of predictions: the smaller the total total loss $\sum_{t=1}^n \Loss_t(p_t)$ is, the better the 
obtained confidence sets. The most straightforward way to achieve this goal is then to use one of the many known online 
algorithms guaranteeing low worst-case regret to generate the sequence $p_1,\dots,p_n$, but this approach has several 
downsides. 
Most notably, algorithms guaranteeing low regret are often computationally intractable, and thus obtaining tight bounds 
may be very expensive (or not affordable at all). This limitation is often significant, in that cheaper alternatives 
may very well result in much worse guarantees (as is the case in the linear regression example described in 
Section~\ref{sec:intro}). This limitation is addressed by the conversion scheme to be described next.

\subsection{Analytic online-to-confidence-set conversions}\label{sec:analytic}
As an alternative to the na\"ive method suggested above, one may opt to define confidence sets of the following simpler 
form instead:
\begin{equation}\label{eq:GLM_conf_set}
 \Theta_n = \ev{\theta\in\real^d: \sum_{t=1}^n \pa{{\ell}_t(\theta) - {\ell}_t(\btheta_n)} \le \beta_n},
\end{equation}
where $\btheta_n \in \real^d$ is a data-dependent \emph{reference parameter}, which may be much cheaper to compute than 
a sequence of well-performing predictions $p_1,\dots,p_n$. An easy choice can be the
MLE $\htheta_n = \min_{\theta\in\real^d} \sum_{t=1}^n \ell_t(\theta)$ (given that it 
exists). For such a confidence set, the main challenge is to choose the confidence-width parameter $\beta_n$ in a way 
that $\Theta_n$ includes the true parameter $\thetastar$ with high probability. Our first main result is the following 
theorem, which shows that a suitable $\beta_n$ can be found via a reduction to regret analysis. 
\begin{theorem}\label{thm:main-reduction}
 Let $\pn=\pa{p_1,\dots,p_n}$ be a sequence of distributions over $\mathcal{Y}$ such that each $p_t$ is 
$\F_{t-1}$-measurable. 
Then, for any $\delta\in(0,1)$, the set defined in Equation~\eqref{eq:GLM_conf_set} satisfies $\PP{\exists n: \thetastar\not\in 
\Theta_n} \le \delta$ with the choice
\[
 \beta_n = \regretp(\btheta_n) + \log \frac 1\delta.
\]
\end{theorem}
\begin{proof}
Fix an online learning algorithm and, for any $\theta \in \real^d$, write
\begin{align*}
\sum_{t=1}^n \pa{\loss_t\pa{\theta} - \ell_t\bpa{\btheta_n}} = 
\underbrace{\sum_{t=1}^n \pa{\Loss_t(p_t) - \ell_t\bpa{\btheta_n}}}_{\regretp\bpa{\btheta_n}} + 
\underbrace{\sum_{t=1}^n \pa{\ell_t(\theta) - \Loss_t(p_t)}}_{M_n(\theta)}\,,
\end{align*}
where we defined $M_n(\theta)$ in the last line, and identified the first sum in the decomposition as the regret of 
the online algorithm against $\btheta_n$. The claim then follows by applying \Cref{prop:generic_lbd} 
to show that 
$\sup_{n \ge 1} M_n(\thetastar) \le \log 
\frac{1}{\delta}$ holds with probability at least $1-\delta$.
\end{proof}

In words, Theorem~\ref{thm:main-reduction} states that if we can upper bound the worst-case regret of an online 
algorithm, then the confidence sequence of Equation~\eqref{eq:GLM_conf_set} will include the true 
parameter $\thetastar$ with high probability. Note that the online algorithm does not need to be executed for 
constructing the confidence set, and one only needs to show the \emph{existence} of an algorithm with low regret for 
the sake of analysis---which is the reason why we refer to this conversion technique as ``analytic''. This addresses 
the computational concerns with the na\"ive sequential likelihood-ratio method we discussed in 
Section~\ref{sec:likelihood-ratio-method}: since there is no need to actually run the online learning algorithm, we can 
choose the method with the best regret bound regardless of how impractical it may be. Furthermore, it is important to 
note that typical regret bounds are guaranteed to hold with probability $1$ for all sequences $(X_t,Y_t)$, and thus the 
conversion scheme suggested by the above theorem cleanly splits the complex statistical problem of designing a suitable 
confidence set into a purely statistical question (bounding a supermartingale via \Cref{prop:generic_lbd}) and a purely 
algorithmic question (bounding the regret).

Not having to run the online algorithm comes with other major conceptual advantages. In particular, the algorithm used 
in this analytic construction may make use of information that would not be available for a less sophisticated 
approach. Most remarkably, the online learner may even have access to $\thetastar$ and use it in making its 
predictions. Other improvements can be achieved when the 
sequence of covariates $X_1,\dots,X_n$ is fixed in advance or is drawn i.i.d., by revealing this sequence to the 
online learner ahead of time (known as \emph{transductive online learning}---see Section~\ref{sec:cases_oblivious}.) 
Finally, note that the tightness of the set is also influenced by the choice of the reference
parameter $\btheta_n$, as can be readily observed after rewriting $\Theta_n$ as
\begin{equation*}
 \Theta_n = \ev{\theta\in\real^d: \sum_{t=1}^n {\ell}_t(\theta) \le \sum_{t=1}^n {\ell}_t(\btheta_n) + 
\regretp(\btheta_n) + \log \frac 1\delta}.
\end{equation*}
Thus, the confidence width is subject to a tradeoff between the total loss of the reference point and the best 
achievable regret against the chosen $\btheta_n$. One can thus restrict the online algorithm to only aim to compete 
with specific choices of $\btheta_n$ that are known to be ``good enough'' in terms of total loss, and enjoy the 
benefits of reduced regret against this restricted comparator. We 
demonstrate the power of this general approach in Section~\ref{sec:cases_analytic}, where we 
use it to recover and tighten a number of state-of-the-art concentration inequalities.

\subsection{Algorithmic online-to-confidence-set conversions}\label{sec:algorithmic}
The analytic conversions we presented in the previous section provide a rigorous alternative to the na\"ive 
sequential likelihood ratio method outlined in Section~\ref{sec:likelihood-ratio-method}. The tightness of the 
resulting bounds depended on the best achievable regret against the data-dependent comparator $\btheta_n$. It is 
natural to ask if it is possible to design confidence sets whose size depends on the best achievable regret against the 
\emph{true parameter} $\thetastar$ instead, to make use of potential additional structure that may be present in the 
problem at hand. In this section, we provide a methodology that achieves this. In what follows, we present a 
simplified recipe for deterministic forecasters that select each $q_t$ as a point mass concentrated on $\theta_t$, 
and provide a tighter (but technically more involved) result in Section~\ref{sec:cases_algorithmic}.

Our algorithmic conversions are stated using the concept of \emph{$\eta$-shifted losses}, defined for any $\eta \in 
(0, 1]$ as $$\shloss_{t}(\theta) = \ell_t\big(\eta\theta + (1-\eta)\theta^\star\big)\,.$$ Clearly, we have that for $\eta=1$ this loss coincides 
with $\ell_t$. We remark that, for any $\eta$, we have that $\shloss_t(\theta^\star) = 
\ell_t(\theta^\star)$. The following proposition gives a result analogous to \Cref{prop:generic_lbd} for the total 
losses of mixture forecasters. Let us also define the associated \emph{shifted log loss} of a mixture forecaster as $$\shLoss_{t}(q) = -\log\int e^{-\shloss_{t}(\theta)}\dd 
q(\theta)\,,$$ where $q\in\Delta_\Theta$. 
For $\eta=1$, this corresponds to the loss $\Ell_t$ for the mixture forecaster defined in Equation \eqref{eq:mixloss}.

\begin{proposition}\label{prop:shifted_lbd}
Let $q^n = (q_1,\dots,q_n)$ be a sequence of distributions on $\Theta$ such that each $q_t$ is $\F_{t-1}$-measurable. 
Then, for any 
$\eta\in (0, 1]$, and any $\delta > 0 $,
\begin{equation*}
\Pp\left[\exists n \geq 1\,:\; \sum_{t=1}^{n}\ell_t(\theta^\star) - 
\sum_{t=1}^{n}\shLoss_{t}(q_t) \geq \log(1/\delta)\right] \leq \delta\,.
\end{equation*}
\end{proposition}
\begin{proof}
First, note that
\begin{align*}
&\EEcc{e^{\ell_n(\theta^\star) - \shLoss_{t}(q_n)}}{\mathcal{F}_{n-1}} = 
\int\EEcc{e^{\ell_n(\theta^\star)-\shloss_t(\theta)}}{\mathcal{F}_{n-1}}\dd q_n(\theta)\\
&\qquad\qquad\qquad= \int\frac{\EEcc{e^{\eta\inner{X_t}{\theta - 
\theta^\star}Y_t}}{\mathcal{F}_{n-1}}}{\exp\left(\psi\big(\inner{\theta^\star}{X_t} + 
\eta\inner{\theta-\theta^\star}{X_t}\big) - \psi\big(\inner{\theta^\star}{X_t}\big)\right)}\dd q_n(\theta)= 1\,,
\end{align*}
where we used that for an exponential family, the equality 
$$\log\mathbb{E}[e^{\beta Y_t}|\mathcal{F}_{t-1}] = \psi(\beta + 
\inner{\theta^\star}{X_t}) - \psi(\inner{\theta^\star}{X_t})$$ holds for any $\mathcal{F}_{t-1}$-measurable $\beta$. 
So, we have proved that for any $t$, we have
\begin{equation*}
\mathbb{E}[M_t|\mathcal{F}_{t-1}] = M_{t-1}\mathbb{E}[\exp(\ell_t(\theta^\star) - 
\shLoss_{t}(q_t))|\mathcal{F}_{t-1}] =  M_{t-1}\,.
\end{equation*} with $M_t = \exp\left(\sum_{s=1}^{t}\ell_s(\theta^\star) - 
\sum_{s=1}^{n}\shLoss_{s}(q_s)\right)$. 
This yields that $M_t$ is a non-negative martingale, and the desired result follows by Ville's inequality.
\end{proof}
Using the above result, we will derive confidence sets in terms of the function
\begin{equation*}
d_\psi(z, z') = \frac{1}{2}\psi(z) + \frac{1}{2}\psi(z') - \psi(z/2 + z'/2).
\end{equation*}
For any $\psi$, $d_\psi$ is symmetric and satisfies $d_\psi(z,z) = 0$ for every $z$. Thanks to the convexity of $\psi$, 
$d_\psi$ is nonnegative (and also positive definite if $\psi$ is strictly convex, i.e., $d_\psi(z, z') \neq 0$ when $z 
\neq z'$). However, in general, $d_\psi$ is not a proper distance because it does not satisfy the triangle inequality. 
Loosely speaking, the rate at which $d_\psi(z, z')$ increases as $z'$ moves further away from $z$ is determined by the 
curvature of $\psi$. 
Our main result of this section provides a confidence set for $\thetastar$ based on the sequence 
$q_1,\dots,q_n$ output by an online learning algorithm, in terms of the function $d_\psi$.

\begin{theorem}\label{thm:algorithmic_general}
Let $\theta^n=\pa{\theta_1,\dots,\theta_n}$ be a sequence of parameters in $\Theta$ such that each $\theta_t$ is 
$\F_{t-1}$-measurable. 
Then, for any $\delta\in(0,1)$,
$$
\PP{\exists n \geq 1\,:\; 
\sum_{t=1}^n d_{\psi}(\inner{\theta_t}{X_t}, \inner{\theta^\star}{X_t}) \leq 
\frac 12 \regret_{\theta^n}(\theta^\star) + \log\frac{1}{\delta}} \le \delta\,.
$$
\end{theorem}
\begin{proof}
Let us introduce the shorthand $D_{t}(\theta) = \frac 12 \ell_t(\theta) + 
 \frac 12 \ell_t(\theta^\star) - \loss_{t}^{(1/2)}(\theta)$, and 
observe that $D_{t}(\theta) = d_{\psi}(\inner{\theta}{X_t}, \inner{\theta^\star}{X_t})$.
Then, we have 
\begin{align*}
\sum_{t=1}^n D_t(\theta_t) = \frac 12 \sum_{t=1}^n \pa{\loss_{t}(\theta_t) - \loss_t(\theta^*)} + 
\sum_{t=1}^n \pa{\loss_{t}(\theta^*) - \loss_t^{(1/2)}(\theta_t)}
 \end{align*}
The claim then follows from applying \Cref{prop:generic_lbd} with $\eta = 1/2$.
\end{proof}
Whereas the previously presented constructions use \emph{either} the predictions of an online 
learning algorithm (Section~\ref{sec:likelihood-ratio-method}) \emph{or} a regret bound (Section~\ref{sec:analytic}), 
this one uses a bound on the regret \emph{and} the predictions of an online learning algorithm. In 
return, the size of the confidence sets is determined by the regret against $\theta^{\star}$ rather than the regret 
against $\hat{\theta}_n$, which can be advantageous when, for instance, $\theta^{\star}$ is known to enjoy some 
additional structure (e.g., sparsity).

Supposing that the online algorithm comes with a regret bound of the form $\regret_{\theta^n}(\theta) \le 
B_n(\theta)$, the result above implies that the following is a valid confidence sequence:
\begin{equation}\label{eq:algorithmic_set}
\Theta_n = \left\{\theta \in \mathbb{R}^d\;:\; \sum_{t=1}^n d_{\psi}(\inner{\theta_t}{X_t}, \inner{\theta}{X_t}) 
\leq 
\frac 12 B_{n}(\theta) + \log\frac{1}{\delta}\right\}\,.
\end{equation}
If a tight upper bound on $\regret_{\theta^n}(\thetastar)$ is known, the comparator-specific 
$B_{n}(\theta)$ term can be replaced by said bound while maintaining the validity of the confidence sequence.
Unfortunately, $d_\psi$ is not always a convex function, which means that the sets in \Cref{thm:algorithmic_general} 
are not always convex sets. If $\psi$ is strongly convex, this can be exploited to construct convex confidence 
ellipsoids that contain the sets in \Cref{eq:algorithmic_set}, though they may be substantially larger. We provide 
concrete 
examples (as well as a tighter confidence set of a similar shape) in Section~\ref{sec:algorithmic}.

\section{Regret bounds}\label{sec:regret}
In order to derive meaningful confidence sets from the framework described in the previous section, we need to 
exhibit sequential prediction algorithms that come with guaranteed bounds on their regret. Recall that our main result 
(\Cref{thm:main-reduction}) only requires demonstrating the \emph{existence} of such methods, without concern 
for computational aspects. This allows us to consider algorithms that are generally hard (or impossible) to implement 
efficiently. For this reason, we will focus on variants of two specific algorithms that come with tight regret 
bounds, but are not necessarily easily implementable: the exponentially weighted average (EWA) and normalized maximum 
likelihood (NML) forecasters. We formally introduce these methods in the sections below, as 
well as provide bounds on their regret, stated in rather general terms. Some results will concern the special case of 
\emph{transductive online learning}, meaning that the sequence of covariates $X_1,X_2,\dots,X_n$ is chosen obliviously 
of the predictions made by the online learner, and that this sequence is known ahead of time. All bounds we provide 
in this section are guaranteed to hold with probability $1$.
We will instantiate these bounds in \Cref{sec:cases} below to provide some concrete confidence sets for 
GLMs. For further reading on sequential prediction with the log loss, we refer to the excellent book of 
\citet{cesa2006prediction} (and particularly Chapter~9 therein).

\subsection{The Exponentially Weighted Average Forecaster} \label{sec:EWA}
One of the most fundamental algorithms for sequential 
probability assignment is the \emph{exponentially weighted average} (EWA) forecaster, first proposed and studied by 
\citet{Vov90} (with later developments by \citealp{LW94,FS97} and numerous further applications throughout all of 
online learning, cf.~\citealp{cesa2006prediction}). This strategy is a mixture forecaster, which takes as 
input a \emph{prior} $q_1$ over the parameter 
space $\Theta$, and then produces each subsequent mixture for $t=2,\dots,n$ according to the update rule
\begin{equation*}
\frac{\dd q_t}{\dd q_1}(\theta) = \frac{e^{-\lambda\sum_{s=1}^{t-1}\ell_s(\theta)}}{\int 
e^{-\lambda\sum_{s=1}^{t-1}\ell_s} \dd q_1}.
\end{equation*}
Here, the parameter $\lambda > 0$ can be interpreted as a \emph{learning rate} (or \emph{stepsize}). The EWA forecaster 
comes with guarantees on its \emph{scaled log loss} with scale parameter $\lambda$ defined for each mixture 
$q\in\Delta_{\Theta}$ as
\[
\mathcal{L}_{t, \lambda}(q) = -\frac{1}{\lambda}\log\left(\int e^{-\lambda\ell_t(\theta)}\dd q(\theta)\right)\,,
\]
with the scaled regret defined as $\regretql{\lambda}(\bar{\theta}) = \sum_{t=1}^n \pa{\mathcal{L}_{t, \lambda}(q_t) 
-\loss_t(\bar{\theta})}$. 
Indeed, through a standard 
telescoping sum argument (cf. \Cref{lem:telescope}), one can show that for any comparator $\bar{\theta}$ and any $\lambda > 0$, the 
regret of the EWA strategy satisfies
\begin{equation}
\regretql{\lambda}(\bar{\theta}) = -\frac{1}{\lambda}\log\left(\int e^{-\lambda\sum_{t=1}^{n}\big(\ell_t(\theta) - 
\ell_t(\bar{\theta})\big)}\dd q_1(\theta)\right).\label{eqn:ewa_regret}
\end{equation}

There are several ways in which one can turn the above generic bound into concrete, quantitative guarantees. A result 
we will repeatedly use 
provides an upper bound in terms of a quantity we call the \emph{Bregman information gain}, defined in terms of the 
so-called Bregman divergence associated 
with the log-likelihood of the GLM we study. For a precise definition, let $\rho: \mathbb{R}^d \to \mathbb{R}$ be any 
convex and differentiable function such that $\int\exp(-\rho(\theta))\mathrm{d}\theta < \infty$ and define the 
regularized cumulative loss 
\[
Z_{n,\lambda}^{\rho}(\theta) = 
\lambda\sum_{t=1}^{n}\ell_t(\theta) + \rho(\theta)\,.
\]
For any convex and  differentiable function $f: \mathbb{R}^d \to 
\mathbb{R}$, define the Bregman divergence as
\begin{equation*}
\mathcal{B}_f(\theta, \theta^{\prime}) = f(\theta) - f(\theta^{\prime}) - \inner{\theta - \theta^{\prime}}{\nabla 
f(\theta^{\prime})}\,,
\end{equation*}
which is the difference between $f(\theta)$ and the first-order Taylor approximation of $f$ around $\theta'$. Then, we 
define the Bregman information gain as
\begin{equation}\label{eq:breg}
\gamma_{n,\lambda}^{\rho} = -\log\left(\frac{\int\exp(-\mathcal{B}_{Z_{n,\lambda}^{\rho}}(\theta, 
\wh{\theta}_{n,\lambda}))\mathrm{d}\theta}{\int\exp(-\rho(\theta))\mathrm{d}\theta}\right)\,,
\end{equation}
where $\wh{\theta}_{n,\lambda} \in \argmin_{\theta}Z_{n,\lambda}^\rho(\theta)$, assuming this exists 
(which will be the case in all applications we consider). The term ``Bregman information gain'' is borrowed from 
\citet{bregman2023}, who justify this naming convention by the observation that in the case of linear models, this 
quantity is equal to the mutual information between the function values $\inner{\theta^{\star}}{X_1}, \dots, 
\inner{\theta^{\star}}{X_n}$ and the labels $Y^n$, which is often thought of as a measure of ``information gain'' about 
the true parameter $\theta^{\star}$ given the observed samples in a Bayesian model.

The following proposition provides an upper bound on the regret of the EWA forecaster in terms of the Bregman 
information 
gain, when using the prior $q_1 \propto e^{-\rho}$.
\begin{proposition}\label{lemma:regrewa}
For all comparators $\bar{\theta}$ such that $\rho(\bar{\theta}) < \infty$, the regret of the EWA forecaster with prior 
density $q_1\propto e^{-\rho}$ satisfies
\begin{equation*}
\regretql{\lambda}(\bar{\theta}) \leq \frac{\rho(\bar{\theta}) + \gamma_{n,\lambda}^{\rho}}{\lambda}\,.
\end{equation*}
\end{proposition}

\begin{proof}
Recalling the expression \eqref{eqn:ewa_regret} of the regret of the EWA forecaster, we write
\begin{align*}
\regretql{\lambda}(\bar{\theta}) &= -\frac{1}{\lambda}\log\left(\int\exp\left(-\lambda\sum_{t=1}^{n}\ell_t(\theta) + 
\lambda\sum_{t=1}^{n}\ell_t(\bar{\theta})\right)\dd q_1(\theta)\right)\\
&= -\frac{1}{\lambda}\log\left(\frac{\int\exp\left(-Z_{n,\lambda}^{\rho}(\theta) + 
 Z_{n,\lambda}^{\rho}(\bar{\theta})\right)\mathrm{d}\theta}{\int\exp\left(-\rho(\theta)\right)\mathrm{d}\theta}\right) 
+ \frac{\rho(\bar{\theta})}{\lambda}\\
&\leq -\frac{1}{\lambda}\log\left(\frac{\int\exp\left(-Z_{n,\lambda}^{\rho}(\theta) + 
Z_{n,\lambda}^{\rho}(\wh{\theta}_{n,\lambda})\right)\mathrm{d}\theta}{\int\exp\left(-\rho(\theta)\right)\mathrm{d}
\theta}\right) + \frac{\rho(\bar{\theta})}{\lambda}\\
&= -\frac{1}{\lambda}\log\left(\frac{\int\exp\left(-\mathcal{B}_{Z_{n,\lambda}^{\rho}}(\theta, 
\wh{\theta}_{n,\lambda})\right)\mathrm{d}\theta}{\int\exp\left(-\rho(\theta)\right)\mathrm{d}\theta}\right) + 
\frac{\rho(\bar{\theta})}{\lambda} = \frac{\gamma_{n,\lambda}^{\rho} + \rho(\bar{\theta})}{\lambda}\,,
\end{align*}
where we used that $Z_{n,\lambda}^\rho(\theta)-Z_{n,\lambda}^\rho(\wh\theta_{n,\lambda}) = 
\Bb_{Z_{n,\lambda}^\rho}(\theta,\wh\theta_{n,\lambda})$, which holds since $\wh\theta_{n,\lambda}$ minimizes 
$Z_{n,\lambda}^\rho$.
\end{proof}
Under mild conditions, the Bregman information gain (and hence this regret bound) grows asymptotically as 
$d\log n$ in the worst case \citep{grunwald2007minimum}. As our focus is on finite-sample bounds, we will derive more 
precise bounds in the forthcoming sections, under concrete assumptions about the likelihood, the parameter 
$\thetastar$ and the sequence of covariates. 
Naturally, the best results are achieved when picking the regularizer $\rho$ (and thus the prior $q_1$) in a way that 
utilizes the problem structure at hand effectively. For instance, $\rho$ may depend on the smoothness of the 
log-likelihood function, the sequence of covariates $X_1,\dots,X_n$ (if these are known in advance) or even 
$\thetastar$. Our applications in \Cref{sec:cases} will make use of a number of such problem-dependent 
choices.

We provide below an additional result: a variation of EWA using a subset-selection prior inspired by \citet{alquier2011pac}, which enjoys regret guarantees that are adaptive to the \emph{sparsity} of the comparator, improving the worst-case dependence of order $d \log n$ to $s \log n$ in the special case where the comparator $\btheta$ is $s$-sparse. We define a distribution $\pi$ over subsets $S \subseteq [d]$ by
\begin{equation*}
\pi(S) = \frac{2^{-|S|}}{\binom{d}{|S|}\sum_{s=0}^{d}2^{-s}}\,.
\end{equation*}
For each subset $S \subseteq [d]$, we let $\Theta_S = \{\theta \in \mathbb{R}^d: \mathrm{supp}(\theta) \subseteq S\}$ and define the probability measure $q_S \in \Delta_{\Theta}$, which has support contained in $\Theta_{S}$ and whose density w.r.t.\ the Lebesgue measure on $\Theta_S$ is $q_S(\theta) = \frac{\exp(-\rho(\theta))}{\int_{\Theta_S}\exp(-\rho(\theta))\mathrm{d}\theta}$. Finally, we construct our prior as $q_1 = \sum_{S \subseteq [d]}\pi(S) q_S$. Our regret bound for the EWA forecaster with this prior depends on a quantity we call the \emph{restricted Bregman information gain}. For any subset $S \subseteq [d]$, we let $\wh\theta_{n,\lambda,S} = \argmin_{\theta \in \Theta_S}\{Z_{n,\lambda}^{\rho}(\theta)\}$. We then define the restricted Bregman information gain as
\begin{equation}
\gamma_{n,\lambda}^{\rho, S} = -\log\left(\frac{\int_{\Theta_S}\exp(-Z_{n,\lambda}^{\rho}(\theta) + Z_{n,\lambda}^{\rho}(\wh\theta_{n,\lambda,S}))\mathrm{d}\theta}{\int_{\Theta_S}\exp(-\rho(\theta)\mathrm{d}\theta}\right)\,.\label{eqn:res_breg_info}
\end{equation}
We refer the curious reader to \Cref{sec:res_breg_info}, where we show that $\gamma_{n,\lambda}^{\rho, S}$ can be 
equivalently defined in terms of a Bregman divergence associated with the log-likelihood (which not only justifies 
the notation and choice of name, but will also come handy when putting this result to use in 
Section~\ref{sec:cases_algorithmic}). For any comparator $\bar{\theta}$, with support $\bar{S}$, \Cref{lem:sparse} 
provides an upper bound on the regret of the 
EWA forecaster that depends on the restricted Bregman information gain $\gamma_{n,\lambda}^{\rho,\bar{S}}$.

\begin{proposition}\label{lem:sparse}
Let $\bar{S} = \supp(\bar{\theta})$. For all comparators $\bar{\theta}$ such that $\rho(\bar{\theta}) < \infty$, the regret of the EWA forecaster with prior $q_1 = \sum_{S\subseteq [d]} \pi(S) q_S$ satisfies
\begin{equation*}
\regretql{\lambda}(\bar{\theta}) \leq \frac{1}{\lambda}\left(\gamma_{n,\lambda}^{\rho,\bar{S}} + \rho(\bar{\theta}) + |\bar{S}|\log(2ed/|\bar{S}|) 
+ \log 2\right)\,.
\end{equation*}
\end{proposition}

To our knowledge, this concrete regret bound has not appeared in previous literature, although it bears a strong 
similarity with a previous result by \citet{gerchinovitz2013sparsity}. Their method is also based on an application of 
the EWA forecaster, although with a different sparsity-inducing prior inspired by \citet{dalalyan2008aggregation,dalalyan2012mirror}.  We provide further comparisons between the two methods in 
Section~\ref{sec:ewa_algorithmic} where we instantiate our regret bound in the context of confidence sets for sparse 
models.

\begin{proof}
The proof is almost the same as the proof of \Cref{lemma:regrewa}. Starting again from  \eqref{eqn:ewa_regret}, we have
\begin{align*}
\regretql{\lambda}(\bar{\theta}) &= -\frac{1}{\lambda}\log\left(\int\exp\left(-\lambda\sum_{t=1}^{n}\ell_t(\theta) + \lambda\sum_{t=1}^{n}\ell_t(\bar{\theta})\right)\dd q_1(\theta)\right)\\
&= -\frac{1}{\lambda}\log\left(\sum_{S \subseteq [d]}\pi(S)\frac{\int_{\Theta_S}\exp\left(-Z_{n,\lambda}^{\rho}(\theta) + Z_{n,\lambda}^{\rho}(\bar{\theta}) - \rho(\bar{\theta})\right)\dd\theta}{\int_{\Theta_S}\exp\left(-\rho(\theta)\right)\dd\theta}\right)\\
&\leq -\frac{1}{\lambda}\log\left(\frac{\int_{\Theta_{\bar S}}\exp\left(-Z_{n,\lambda}^{\rho}(\theta) + Z_{n,\lambda}^{\rho}(\bar{\theta})\right)\dd\theta}{\int_{\Theta_{\bar S}}\exp\left(-\rho(\theta)\right)\dd\theta}\right) + \frac{\rho(\bar{\theta}) + \log\frac{1}{\pi(\bar{S})}}{\lambda}\\
&\leq \frac{\gamma_{n,\lambda}^{\rho, \bar S} + \rho(\bar\theta) + \log\frac{1}{\pi(\bar S)}}{\lambda}\,,
\end{align*}
where the first inequality follows from the fact that for any non-negative mapping $f$ we have $\sum_S \pi(S)f(S)\geq 
\pi(\bar S) f(\bar S)$. Now we are just left with controlling $\log\frac{1}{\pi(\bar S)}$. By the construction of 
$\pi$, we have $\frac{1}{\pi(\bar S)} \leq  \binom{d}{\bar S}2^{1+|\bar S|}$, from which the claim follows by using that 
$\binom{d}{s} \leq (ed/s)^s$.
\end{proof}

\subsection{Normalized Maximum Likelihood}\label{sec:NML}
The second family of methods we will consider is that of \emph{Normalized Maximum Likelihood} (NML) forecasters, first 
proposed by \citet{Shtarkov1987,Rissanen1996} and later studied extensively in a long sequence of works including  
\citet{BarrenRissanenYu1998,TakeuchiB97,XieB00,CesaBianchiLugosi2001,LianFengBarron2006,pmlr-v30-Bartlett13,
GrunwaldHarremoes2009,pmlr-v98-grunwald19a,jacquet2022precise},
and in 
particular the excellent book of \citet{grunwald2007minimum}.
Rather than playing distributions on $\Theta$ in each round, NML forecasters work directly with distributions on 
$\mathcal{Y}$ in each round. We recall that for a generic $p\in\Delta_\mathcal Y$, the log loss is defined as 
$\mathcal{L}_t(p) = -\log(p(Y_t))$, and the regret associated with a sequence $\pn = p_1, \dots, p_n$ as
\begin{equation*}
\regret_{\pn}(\bar{\theta}) = \sum_{t=1}^{n}\mathcal{L}_t(p_t) - \sum_{t=1}^{n}-\log(p(Y_t|X_t, \bar{\theta}))\,.
\end{equation*}

The standard Normalized Maximum Likelihood (NML) forecaster \citep{Shtarkov1987} is defined in terms of the joint 
distribution over sequences in $\Y^n$ with density defined as
\begin{equation}
 P_n(y^n) = \frac{\sup_{\theta \in \Theta}\{\prod_{t=1}^{n}p(y_t|X_t, 
\theta)\}}{\int_{\mathcal{Y}^n} \sup_{\theta \in \Theta}\{\prod_{t=1}^{n}p(\wt{y}_t|X_t, 
\theta)\}\dd \wt{y}^n }\,,\label{eqn:nml}
\end{equation}
whenever the normalization constant can be guaranteed to be finite. Since this is typically not the case in the setting 
we consider\footnote{For example, in the case of linear regression with $n=1$,
$p(y_1|X_1,\theta)$ is maximized by any $\theta$ satisfying $\langle \theta, X_t\rangle = y_t$;
as a result, $\sup_{\theta} p(y_1|X_1,\theta) = (2\pi)^{-1/2}$ does not depend on $y_1$, and thus cannot be normalized.}, the standard definition can be modified in a number of ways 
that lead to well-defined probability distributions. Our of the many possibilities\footnote{For alternative techniques 
leading to well-defined distributions, we refer to Chapter~11 of \citealp{grunwald2007minimum}, where this method is 
called ``LNML-2''. For simplicity, we will refer to this method as NML below.}, and in line with the 
spirit of the EWA forecaster described previously, we address this here by introducing a ``prior'' (or sometimes called 
``luckiness function'' in the NML literature) $\rho:\Theta\ra\real$, and considering the following distribution over 
sequences in $\Y^n$:
\begin{equation}
 P_n(y^n) = \frac{\sup_{\theta \in \Theta}\{\prod_{t=1}^{n}p(y_t|X_t, 
\theta)e^{-\rho(\theta)}\}}{\int_{\mathcal{Y}^n} \sup_{\theta \in 
\Theta}\{\prod_{t=1}^{n}p(\wt{y}_t|X_t, 
\theta)e^{-\rho(\theta)}\} \dd \wt{y}^n}\,.\label{eqn:nml2}
\end{equation}
Having defined this joint distribution, the NML strategy consists in playing the conditional distributions extracted 
from $P_n$ via the formula
\begin{equation*}
p_t(y_t) = 
\frac{\int_{\mathcal{Y}^{n-t}}P_n(y^n)\mathrm{d}y_{t+1}\cdots\mathrm{d}y_n}{\int_{\mathcal{Y}^{n-t+1}}P_n(y^n)\mathrm{d} 
y_{t}\cdots\mathrm{d}y_n}\,.
\end{equation*}
Notice that each $p_t$ depends on the entire sequence of covariates $X^n$, which means the NML forecaster as stated is 
only suitable for the transductive setting where all covariates are known ahead of time. 

The denominator in the expression of Equation~\eqref{eqn:nml2} plays a special role in the regret analysis of NML. 
Following the conventions in the literature, we call this quantity the \emph{Shtarkov sum} \citep{grunwald2007minimum}, and denote it by
\begin{equation}\label{eq:shtarkov1}
\mathcal{S}(\Theta, \rho, X^n) = \int_{\mathcal{Y}^n} \sup_{\theta \in \Theta}\left\{\prod_{t=1}^{n}p(y_t|X_t, 
\theta)e^{-\rho(\theta)}\right\} \dd y^n.
\end{equation}
Within the transductive setting, the NML forecaster can be shown to be the unique minimax optimal forecaster for 
 sequential probability assignment, and for any sequence $X^n$, the minimax regret is equal to the 
logarithm of the Shtarkov sum (up to the regularization due to the prior $\rho$). Indeed, the worst-case 
regularized regret of the NML forecaster can be written as
\begin{align*}
\sup_{\theta \in \Theta}\{\regret_{\pn}(\theta)-\rho(\theta)\} &= \sum_{t=1}^{n}\mathcal{L}_t(p_t) + \sup_{\theta 
\in \Theta}\left\{\sum_{t=1}^{n}\log(p(Y_t|X_t, \theta)) - \rho(\theta)\right\}\\
&= \log\left(\frac{1}{p_n(Y^n)}\right) + \log\left(\sup_{\theta \in \Theta}\left\{\prod_{t=1}^{n}p(Y_t|X_t, 
\theta)e^{-\rho(\theta)}\right\}\right)\\
&= \log(\mathcal{S}(\Theta, \rho, X^n))\,.
\end{align*}
We note that $\mathcal{S}(\Theta, \rho, X^n)$ does not depend on $Y^n$, which means that the NML forecaster is an 
\emph{equalizer}: it achieves the same regret on every realization of the labels $Y^n$. Together with the the fact that 
any forecaster induces a probability distribution $p_n$ on $\mathcal{Y}^n$ (and vice versa), this can be used to show 
that the NML forecaster is the unique minimax optimal forecaster in terms of the regularized regret (see, e.g., 
Section~9.4 in \citealp{cesa2006prediction} and Section~11.3 in \citealp{grunwald2007minimum}). There exist variations 
of the standard NML forecaster that are able to deal with adaptively chosen covariates as well, including the Sequential 
Normalized Maximum Likelihood (SNML) forecaster \citep{RoosRissanen08,KotGru11}, and the contextual NML (cNML) forecaster of 
\citet{liu2024sequential} which enjoys a minimax-optimality property similar to the one satisfied by NML. 

While it is generally hard to evaluate the Shtarkov sum in concrete settings of interest, there are some important 
special cases where this can be done and the minimax regret be evaluated. The most prominent example falling into the 
scope of the present paper is that of linear models (i.e., the setting we consider with the choice $\psi: z\mapsto 
z^2/2$). In this setting, the sequence of predictions produced by NML as defined above with $\rho: \theta \mapsto 
\frac{1}{2\gamma^2}\twonorm{\theta}^2$ coincide with those of the EWA forecaster with prior $q_1 \propto e^{-\rho}$ 
(see \citealp{kakade2005worst} and Section~11.3 in \citealp{grunwald2007minimum}). This shows that EWA is minimax 
optimal in the setting where the covariates $X^n$ are known ahead of time. Curiously, EWA requires no prior knowledge of 
the sequence of covariates (and in fact is also equivalent to SNML in this case), implying that there is no gap in 
difficulty of the fixed- and adaptive-design models in this setting. Accordingly, the regret bounds of EWA we will make 
use of below are all minimax optimal for linear models with fixed and obliviously chosen covariates. For several other 
settings, the Shtarkov sum is known to grow asymptotically as $\Theta\bpa{\frac{d}{2}\log n}$ 
\citep{jacquet2022precise}. Since our aim in this paper is to derive explicit finite-sample guarantees, we will not 
instantiate these asymptotic results below. Finally, several of the bounds we provide will show explicit dependence on 
the sequence covariates, leading to rates that are potentially better than minimax in benign cases (e.g., when the 
sequence of covariates does not span the full space).

\section{Applications}\label{sec:cases}
In this section, we instantiate our online-to-confidence-set technique in a number of specific cases of interest. 
Throughout the section, we will assume that the log-partition function of the GLM is \emph{$M$-smooth} in the sense 
that $\psi$ is twice-differentiable with its second derivative satisfying $\psi''(z)\le M$ for some positive $M$ and 
all $z\in\real$. Many GLMs of practical interest satisfy this condition, including the classic linear model $\psi(z) = 
\frac {z^2}{2} $ with $M = 1$ and the logistic model $\psi(z) = \log(1 + e^z)$ with $M=\frac 14$ \footnote{For $\psi(z) 
= \log(1 + e^z)$ and $\mu(z) = 1/(1 + e^{-z}) \in [0,1]$, $\psi^{\prime\prime}(z) = \mu(z)(1 - \mu(z)) \in [0, 
\frac{1}{4}]$.}. We note that this condition is equivalent to assuming that the random variable $Y_t - \EEcc{Y_t}{\mathcal{F}_{t-1}}$ 
is $\sqrt{M}$-subgaussian\footnote{This follows from observing that the centred moment generating function satisfies 
$\EEcc{\exp(\lambda (Y_t-\E[Y_t|\mathcal{F}_{t-1}]))}{\mathcal{F}_{t-1}} = \Bb_\psi(\gamma + \lambda\|\gamma) \le M \lambda^2/2$.}. 

Additionally, some of the results below will also assume that $\psi$ is \emph{$m$-strongly convex} on an interval 
$[-b,b]$, meaning that 
$\psi''(z) \ge m$ holds uniformly for all $z\in[-b,b]$ and some $m \ge 0$ (and some $b>0$). Obviously, $m \le M$ holds 
for all GLMs, and the two are equal if and only if $\psi$ is quadratic (i.e., the GLM is linear). We will sometimes call 
the ratio of the two constants the \emph{condition number}\footnote{This is not to be confused with the condition number 
of the covariance matrix associated with the data, which does not appear in any of our bounds.} of the GLM and 
denote it by $\kappa = M/m$. We highlight that assuming strong convexity of $\psi$ is generally a very strong 
assumption, and the constant $m$ might often scale poorly with problem parameters such as the dimension $d$. For 
instance, for logistic regression, the strong convexity assumption only holds whenever $\Theta$ is compact and the 
covariates are all bounded, and even then $m$ can be exponentially small with $d$. Thus, the guarantees stated below 
without assuming strong convexity are to be regarded as much less restrictive than the ones requiring this condition.

Most results we show below are derived from the analytic conversion scheme presented in 
\Cref{sec:analytic} and the confidence sets have the shape given in Equation~\eqref{eq:GLM_conf_set}. For each 
of these applications, our technique yields the best known upper bounds on the width parameter $\beta_n$. Later 
results also make use of the algorithmic conversion scheme of \Cref{sec:algorithmic}, which also leads to 
improved confidence sets in comparison with previous work. We provide a detailed discussion of the relevant literature 
after stating each result, and point out the concrete improvements explicitly.

\subsection{Analytic conversions} \label{sec:cases_analytic}
To apply the analytic conversion scheme suggested by \Cref{thm:main-reduction}, one needs to find a suitable 
reference point $\btheta_n$ and demonstrate the existence of an algorithm with low regret. In what follows, we consider 
a number of concrete settings, where we demonstrate useful choices of these hyperparameters and present the resulting 
confidence sets. 

\subsubsection{Adaptively chosen covariates}\label{sec:cases_adaptive}
We first consider the most general version of our setting, where each covariate $X_t$ is allowed to depend on the 
previous sequence of outcomes $\pa{X_k,Y_k}_{k < t}$ in an arbitrary fashion. Allowing such dependences is extremely 
important in applications of high practical interest, for instance in sequential decision-making problems such as 
online learning in (generalized) linear bandits or Markov decision processes \citep{lattimore2020bandit}.  
Our main result for this setting is the following theorem.
\begin{theorem}\label{cor:smooth}
Suppose that $\psi$ is $M$-smooth, and fix $\gamma > 0$. Set 
$\rho(\theta) = \frac{\|\theta\|^2_2}{2\gamma^2}$ and let  $\btheta_n = \htheta_n = \argmin_{\theta} 
\sum_{t=1}^n \ell_t(\theta) + \rho(\theta)$. 
Then, for any $\delta\in(0,1)$, the set defined in Equation~\eqref{eq:GLM_conf_set} satisfies $\PP{\exists n: 
\thetastar\not\in \Theta_n} \le \delta$ with the choice
\begin{equation}
\beta_n = \frac{\bnorm{\htheta_n}_2^2}{2\gamma^2} + 
\frac{1}{2}\log{\det\pa{\gamma^2 M\Lambda_n + \Id}} + \log \frac 
1\delta \,.\label{eqn:ewa_conf_width_smooth}
\end{equation}
In particular, if all the covariates are bounded as $\|X_t\|_2\leq L$, we have
\begin{equation}
\beta_n \leq  \frac{\bnorm{\htheta_n}_2^2}{2\gamma^2} + 
\frac{d}{2}\log\left(1 + \frac{\gamma^2 ML^2n}{d}\right) + \log \frac 
1\delta \,.\label{eqn:ewa_regret_smooth_wc}\end{equation}
\label{lemma:regrewa_smooth}
\end{theorem}
The second claim is a simple consequence of the first one, by the inequality of arithmetic and geometric means applied 
in the form 
$$\det(\gamma^2 M\Lambda_n + \Id) \leq \left(1 + \frac{\gamma^2 MnL^2}{d}\right)^d.$$
The main claim follows from instantiating the reduction scheme of \Cref{thm:main-reduction} with the online 
learning algorithm choosen as the EWA forecaster described in \Cref{sec:EWA}. Specifically, the proof follows 
immediately from  applying \Cref{lemma:regrewa} along with the following upper bound on the Bregman information gain 
for smooth GLMs.
\begin{lemma}\label{lemma:smoothinfo}
    Suppose that $\psi$ is $M$-smooth, fix $\gamma, \lambda > 0$ and set $\rho(\theta) = 
\frac{\|\theta\|^2_2}{2\gamma^2}$. Letting $\Lambda_n = \sum_{t=1}^n X_tX_t\transpose$, the Bregman 
information gain satisfies
    $$\gamma_{n,\lambda}^{\rho} \leq \frac{1}{2}\log\det\pa{\lambda M\gamma^2\Lambda_n+\Id}\,.$$
\end{lemma}
\begin{proof}
    The smoothness of $\psi$ implies that $\text{Hess}[Z^\rho_{n,\lambda}]\preceq \lambda M \Lambda_n +\frac{\Id}{\gamma^2}$, and so $$\Bb_{Z^\rho_{n,\lambda}}(\theta, \theta')\leq \frac{1}{2\gamma^2}\|\theta-\theta'\|^2_{\lambda\gamma^2M\Lambda_n+\Id}\,.$$
Thus, the Bregman information gain can be upper-bounded in terms of a Gaussian integral, which can be evaluated via 
straightforward calculations to give
$$\gamma_{n,\lambda}^{\rho} \leq 
-\log\left(\frac{\int\exp\bpa{-\frac{1}{2\gamma^2}\|\theta-\widehat\theta_{\lambda,n}\|^2_{\lambda\gamma^2M\Lambda_n+\Id
} }
\mathrm{d}\theta}{\int\exp(-\frac{1}{2\gamma^2}\|\theta\|^2_2)\mathrm{d}\theta}\right) =\frac{1}{2}\log\det(\lambda 
M\gamma^2\Lambda_n+\Id)\,.$$
This concludes the proof.
\end{proof}

\paragraph{Comparison with state of the art.}
To ease discussion of the above result, it is more practical to rewrite the confidence set in terms of the 
\emph{regularized losses} $\tloss_t = \loss_t + \rho$ as follows:
\begin{equation}\label{eq:GLM_conf_set_regularized}
 \Theta_n = \ev{\theta\in\real^d: \sum_{t=1}^n \pa{{\tloss}_t(\theta) - {\tloss}_t(\htheta_n)} \le \wt{\beta}_n}.
\end{equation}
With this notation, the main claim of \Cref{cor:smooth} can be rewritten as a bound on the confidence-width parameter 
$\wt{\beta}_n$ as 
\begin{align*}
\wt{\beta}_n &= \frac{\bnorm{\thetastar}_2^2}{2\lambda\gamma^2} + 
\frac{1}{2\lambda}\log{\det\pa{\gamma^2\lambda M\Lambda_n + \Id}} + \log \frac 
1\delta \\
&\leq  \frac{\bnorm{\thetastar}_2^2}{2\lambda\gamma^2} + 
\frac{d}{2\lambda}\log\left(1 + \frac{\gamma^2\lambda ML^2n}{d}\right) + \log \frac 
1\delta \,.\end{align*}
If it is known in advance that $\norm{\thetastar}_2 \leq B$, then the choice $\gamma = B$ yields
\begin{equation*}
\wt{\beta}_n \leq \frac{1 + \log(\det(B^2\lambda M\Lambda_n + \Id))}{2\lambda} \leq \frac{1}{2\lambda}\left(1 + 
d\log\left(1 + \frac{B^2\lambda ML^2n}{d}\right)\right)\,.
\end{equation*}
In the special case of linear models, this recovers the classic results of \citet{abbasi2011improved} (see also 
\citealp{dlPLS09} and \citealp{flynn2023improved}). For GLMs, the most directly comparable result is due to 
\citet{lee2024unified}, who work under the slightly more restrictive assumption that the parameter set is compact, and 
provide confidence sets with width as given in our Equation~\eqref{eqn:ewa_regret_smooth_wc} (which is a looser version 
of the width bound guaranteed by our result, given in Equation~\eqref{eqn:ewa_conf_width_smooth}). All known previous 
bounds from the literature suffer from additional constant factors such as a uniform lower bound on the second 
derivative $\psi''$ \citep{jun2017scalable,li2017provably,emmenegger2023likelihood}. For the special case of 
parameter estimation of exponential-family distributions (i.e., GLMs with fixed covariates), our bound recovers the 
confidence sequence proposed by \citet{bregman2023}, with width dependent on the Bregman information gain (which name 
was in fact coined in said work).

To the best of our knowledge, the data-dependent regret bound implied by \Cref{lemma:smoothinfo} does not appear 
explicitly in any prior work on EWA. In the worst case, it recovers a classic regret bound of \citet{kakade2004online} 
for the EWA forecaster under the same boundedness and smoothness assumptions and with the same prior. However, our 
data-dependent bound can be much smaller in practice, especially if $L$ is a loose upper bound on the norm of the 
largest covariate. In addition, the data-dependent regret bound is adaptive to certain easy sequences of covariates. For 
instance, the data-dependent regret bound can be used to obtain
\begin{equation}\label{eq:regret_rank}
\regretq(\bar{\theta}) \leq \frac{\|\bar{\theta}\|_2^2}{2\lambda\gamma^2} + 
\frac{\mathrm{rank}(\Lambda_n)}{2\lambda}\log\left(1 + \frac{\gamma^2ML^2n}{\mathrm{rank}(\Lambda_n)}\right).
\end{equation}
Evidently, this improvement is inherited by our confidence width parameter as well, demonstrating a clear improvement 
over the best previously known results of \citet{lee2024unified}. This fact follows from a small modification of the 
determinant-trace inequality in Lemma~10 of \citet{abbasi2011improved}, which accounts for the fact that $\Lambda_n$ 
may not have full rank (cf. \Cref{lem:det_tr}).

\subsubsection{Obliviously chosen covariates}\label{sec:cases_oblivious}
The guarantees provided in the previous section can be tightened by making a stronger assumption about the sequence of 
covariates: that each $X_t$ is chosen independently of the realized labels $Y_{k}$ (for all $k\neq t$). Some 
well-studied special cases are the \emph{fixed-design} setting where the sequence of covariates is arbitrary but fixed 
before the labels $Y_t$ are drawn, and the \emph{i.i.d.}~setting where each $X_t$ is independently from some fixed 
distribution, independently of all labels and the other covariates. More generally, the set of covariates can be 
drawn from any joint distribution as long as it is independent of the realized labels. 

The lack of dependence between covariates and labels allows us to use online learning algorithms that have prior access 
to the sequence of covariates---a setting that is thoroughly studied in the literature under the name 
\emph{transductive online learning} or \emph{sequential prediction with transductive priors}. Since this setting only 
makes sense for sequences of fixed length, the results we prove below will naturally hold for a fixed sample size $n$. 
Without loss of generality\footnote{If this assumption does not hold without preprocessing, we can work in the 
subspace spanned by the covariates and aim to estimate the projection of $\thetastar$ to said space. Obviously, 
estimating the orthogonal component is impossible in such a situation.}, we will assume that the matrix $\Lambda_n = 
\sum_{t=1}^n X_tX_t\transpose$ is full rank.

The result we will state below will assume that $\psi$ is globally $M$-smooth on $\real$ and locally $m$-strongly 
convex on $[-b,b]$, and furthermore we will suppose that $|\inner{\theta^{\star}}{X_t}| \leq b$ holds. 
For the given sequence of covariates $X_1, \dots, X_n$, we will define the associated \emph{polar set} 
(at scale $b$) as $\Sw_{n,b} = \{\theta \in \mathbb{R}^d: \max_{t \in [n]}|\inner{\theta}{X_t}| \leq b\}$, and 
note that this is a convex set that is guaranteed to include $\thetastar$. For the comparator in the regret analysis, 
we will use the constrained MLE $\wh \theta_{n,b} = \argmin_{\theta \in \Sw_{n,b}}\sum_{t=1}^{n}\ell_t(\theta)$.
The following theorem is our main result about smooth and strongly convex GLMs in this setting, which follows from 
instantiating \Cref{thm:main-reduction} with a transductive online algorithm.
\begin{theorem}\label{thm:transductive}
Suppose that $\thetastar$ satisfies $|\inner{\theta^{\star}}{X_t}| \leq b$ and that 
that $\psi$ is $M$-smooth on $\real$ and $m$ strongly convex on $[-b,b]$, and denote the condition number by $\kappa 
= \frac{M}{m}$. Let $\htheta_{n,b} = 
\argmin_{\theta \in \Sw_{n,b}} \sum_{t=1}^n \ell_t(\theta)$, define $\Psi:\Theta \ra \real$ as $\Psi(\theta) = 
\sum_{t=1}^n \psi\bpa{\iprod{X_t}{\theta}}$ for all $\theta$, and let $\mathcal{B}_{\Psi}$ denote the associated 
Bregman divergence. Then, for any $\delta > 0$, the set defined as
\begin{equation}
 \Theta_n = \ev{\theta: \mathcal{B}_{\Psi}\bpa{\theta,\htheta_{b,n}} \le d \log\pa{1 + 2\kappa} + 2\log \frac 
1\delta }\label{eqn:transductive_width}
\end{equation}
satisfies $\PP{\thetastar\not\in \Theta_n} \le \delta$.
\end{theorem}
Notably, the width of the confidence set above does not show \emph{any} dependence on the sample size $n$, and in 
particular it removes the $\log n$ factor that appeared in the previously stated guarantees for adaptively chosen 
covariates. Also, the bound is completely independent of the realization of the covariate sequence, and in particular 
is invariant to linear transformations of the coordinates of $\thetastar$ and $X^n$. We contextualize this improvement 
in more detail below, after presenting the proof. At a high level, the construction for the proof involves executing 
EWA with a prior that takes into account the complete sequence of covariates, the true parameter $\thetastar$, and the 
log-partition function $\psi$.
\begin{proof}
The proof is based on instantiating the regret bound of EWA executed with $\rho(\theta) = 
\frac{1}{\gamma^2} \norm{\theta - \thetastar}_{\Lambda_n}^2$. A crucial observation is that, thanks to the strong convexity of $\psi$ on $[-b,b]$, the function $\Psi$ is strongly convex on $\Sw_{n,b}$, with respect to the weighted norm $\norm{\cdot}_{\Lambda_n}$, in the following sense:
\begin{equation}\label{eqn:quad_bregman}
\forall \theta, \theta^{\prime} \in \Sw_{n,b}, ~\frac{m}{2}\|\theta - \theta^{\prime}\|_{\Lambda_n}^2 \leq \mathcal{B}_{\Psi}(\theta, \theta^{\prime})\,.
\end{equation}
We define $Z_{n,\gamma}(\theta) = \sum_{t=1}^{n}\ell_t(\theta) + \frac{1}{2\gamma^2}\|\theta - \theta^{\star}\|_{\Lambda_n}^2$ and $\wh \theta_{n,\gamma} = \argmin_{\theta \in \mathbb{R}^d}Z_{n,\gamma}(\theta)$. Due to the smoothness of $\psi$ on $\real$, the Bregman divergence induced by $Z_{n,\gamma}$ can be upper bounded by a quadratic function of $\theta$. In particular, for any $\theta, \theta^{\prime} \in \mathbb{R}^d$,
\begin{align*}
\mathcal{B}_{Z_{n,\gamma}}(\theta, \theta^{\prime}) 
&= \mathcal{B}_{\Psi}(\theta, \theta^{\prime}) + \frac{1}{2\gamma^2}\|\theta - \theta^{\prime}\|_{\Lambda_n}^2 \leq 
\frac{1}{2}(M + 1/\gamma^2)\|\theta - \theta^{\prime}\|_{\Lambda_n}^2\,.
\end{align*}
Now, using the generic EWA regret bound in \Cref{lemma:regrewa} 
(with $\lambda = 1$ and the prior described above) and applying the upper bound given above, we get
\begin{align*}
\regret_{q^n}(\wh \theta_n) &\leq -\log\left(\frac{\int\exp(-\mathcal{B}_{Z_{n,\gamma}}(\theta, \wh 
\theta_{n,\gamma}))\mathrm{d}\theta}{\int\exp(-\frac{1}{2\gamma^2}\|\theta - 
\theta^{\star}\|_{\Lambda_n}^2)\mathrm{d}\theta}\right) + \frac{1}{2\gamma^2}\|\wh \theta_n - 
\theta^{\star}\|_{\Lambda_n}^2\\
&\leq -\log\left(\frac{\int\exp(-\frac{1}{2}(M + \frac{1}{\gamma^2})\|\theta - 
\wh \theta_{n,\gamma}\|_{\Lambda_n}^2)\mathrm{d}\theta}{\int\exp(-\frac{1}{2\gamma^2}\|\theta - 
\theta^{\star}\|_{\Lambda_n}^2)\mathrm{d}\theta}\right) + \frac{1}{m\gamma^2}\mathcal{B}_{\Psi}(\theta^{\star}, \wh 
\theta_n)\\
&= \frac{d}{2}\log(1 + \gamma^2M) + \frac{1}{m\gamma^2}\mathcal{B}_{\Psi}(\theta^{\star}, \wh\theta_n)\,,
\end{align*}
where the last step follows from evaluating the Gaussian integral appearing in the second line. Along with \Cref{thm:main-reduction}, this implies that
\begin{equation}\label{eq:loss_diff_bound}
\sum_{t=1}^{n}\pa{\ell_t(\theta^{\star}) - \ell_t(\wh\theta_{n,b})} \leq \frac{d}{2}\log(1 + \gamma^2M) + 
\frac{1}{m\gamma^2}\mathcal{B}_{\Psi}(\theta^{\star}, \wh\theta_{n,b}) + \log\frac{1}{\delta}\,
\end{equation}
holds with probability at least $1-\delta$. To proceed from here, notice that by the first-order optimality condition 
on the set $\Theta_{n,b}$, we have that $\htheta_{n,b}$ satisfies
$
\siprod{\theta - \htheta_{n,b}}{{\textstyle\sum_{t=1}^{n}}\nabla\ell_t(\htheta_{n,b})} \geq 0
$
for all $\theta \in \Sw_{n,b}$, and in particular for $\thetastar$ too. Therefore,
\begin{equation*}
\mathcal{B}_{\Psi}(\theta^{\star}, \htheta_{n,b}) \leq \sum_{t=1}^{n}\pa{\ell_t(\theta^{\star}) - 
\ell_t(\htheta_{n,b})}\,.
\end{equation*}
Combining this with the bound of Equation~\eqref{eq:loss_diff_bound}, we get
\begin{equation*}
\left(1 - \frac{1}{m\gamma^2}\right)\mathcal{B}_{\Psi}(\theta^{\star}, \htheta_{n,b}) \leq \frac{d}{2}\log(1 + \gamma^2M) 
+ \log\frac{1}{\delta}\,,
\end{equation*}
Picking $\gamma^2 = 2/m$ gives
\begin{equation*}
\mathcal{B}_{\Psi}(\theta^{\star}, \htheta_{n,b}) \leq d\log(1 + 2\kappa) + 2\log(1/\delta)\,,
\end{equation*}
which concludes the proof.
\end{proof}

\paragraph{Comparison with state of the art.} We start by noting that the confidence set given in 
\Cref{thm:transductive} is generally not equal to the set defined in Equation~\eqref{eq:GLM_conf_set}, due to the 
restriction of $\wh{\theta}_{n,b}$ to the polar set $\Sw_{n,b}$. Note however that the two sets coincide if the MLE 
lies within $\Sw_{n,b}$, which can be verified empirically when evaluating the confidence set. For the special case of 
linear regression, $b$ can be clearly set as $+\infty$ and the condition number becomes $\kappa = 1$. In this case, our 
result recovers a classic result of \citet{Cochran1934}
with a slightly worse constant.  For other GLMs, we are 
not aware of any comparable result, and in particular we believe that our confidence set is the first to have a 
width independent of $n$. Furthermore, our construction
property correctly accounts for the parametrization-invariance of GLMs: when transforming all covariates by the 
invertible linear map $A$ and the true parameter $\thetastar$ by $A^{-1}$, the distribution of labels remains 
unchanged. Since the polar set $\Sw_{n,b}$ used in our construction retains this invariance property, our confidence 
set also remains invariant to such reparametrizations. Accordingly, the scale of the covariates does not appear in the 
confidence width either. This phenomenon has been thoroughly studied in the context of transductive online learning, 
where several works have shown that knowing the covariates ahead of time can enable proving scale-invariant regret 
bounds \citep{pmlr-v98-gaillard19a,qian2024refined}. Additionally, these factors have been shown to be impossible to remove without prior knowledge of 
the covariate sequence \citep{Kotlowski2020TCS,foster2018logistic,qian2024refined}.

\subsection{Algorithmic conversions}\label{sec:cases_algorithmic}
We now turn to deriving confidence sequences based on the algorithmic conversion scheme proposed in 
\Cref{thm:algorithmic_general}. We recall that this conversion scheme produces confidence sets whose size depends on the 
regret of an online algorithm against the true parameter $\theta^{\star}$. This allows us to derive tighter confidence 
sets when $\theta^{\star}$ has additional structure. To demonstrate this, the focus of this section is mainly on results 
for \emph{sparse} GLMs, where $\theta^{\star}$ contains mostly zeros. In particular, we assume that $\theta^{\star}$ is 
$s$-sparse for some known $s$ (i.e.\,, that $\sum_{i=1}^d \mathbb{I}\{\theta_i \neq 0\} = s$). Throughout this section, 
we focus solely on the setting where each covariate $X_t$ can depend on the previous observations $(X_k, 
Y_k)_{k=1}^{t-1}$ in an arbitrary fashion.

As mentioned in \Cref{sec:algorithmic}, the function $d_{\psi}$ that appears in the 
algorithmic conversion of \Cref{thm:algorithmic_general} is generally not convex, which means that the 
resulting confidence sets may not be convex either. To address this issue, we assume below that the log-partition 
function $\psi$ is $m$-strongly convex on the interval $[-b,b]$, which implies that $d_\psi(z,z') \ge 
\frac{m(z-z')^2}{8}$ for all $z, z^{\prime} \in [-b, b]$. In particular, this will allow us to derive convex confidence sets that can be stated conveniently in terms of a well-chosen set of 
\emph{pseudo-labels}. Indeed, defining the truncation operator $[z]_b = \max\ev{\min\ev{z,b}-b}$ and the pseudo-label 
$\wh{Y}_t = \left[\iprod{\theta_t}{X_t}\right]_b$, the confidence set of 
\Cref{thm:algorithmic_general} can be shown to be included in the convex set
\begin{equation}\label{eq:det_alg_set}
\Theta_n = \left\{\theta \in \mathbb{R}^d\,:\, \frac 12 \sum_{t=1}^{n}\pa{\iprod{\theta}{X_t} - \wh{Y}_t}^2 \leq 
\beta_n\right\}\,
\end{equation}
for some appropriately chosen $\beta_n$.

To emphasize the improvement of our new algorithmic conversions upon similar existing results, which use deterministic 
forecasters, we first focus on the generic algorithmic conversion scheme for deterministic forecasters in 
\Cref{thm:algorithmic_general}. This alone allows us to recover tighter (by constant factors) algorithmic conversions 
than those in \citet{abbasi2012online} and \citet{jun2017scalable}. However, since all of these conversions use 
deterministic forecasters, the resulting confidence sets have sub-optimal confidence width by at least a factor of 
$\log(n)$. In \Cref{sec:ewa_algorithmic} we address this issue with a tighter algorithmic conversion that is tailored to 
the EWA forecaster.

\subsubsection{Algorithmic Conversions for Deterministic Forecasters}
\label{sec:det_algorithmic}
We start with a straightforward application of \Cref{thm:algorithmic_general} for the case of strongly convex 
log-likelihoods.
\begin{theorem}
Suppose that $\thetastar$ satisfies $|\inner{\theta^{\star}}{X_t}| \leq b$ and that 
$\psi$ is $m$-strongly convex on $[-b,b]$. Then, for any $\delta\in(0,1)$, the set defined in 
Equation~\eqref{eq:det_alg_set} with $\wh{Y_t} = [\iprod{\theta_t}{X_t}]_b$ satisfies $\PP{\exists n: 
\thetastar\not\in \Theta_n} \le \delta$ with the choice
\begin{equation*}
\beta_n = \frac{2}{m}\,\regret_{\theta^n}(\theta^{\star}) + \frac{4}{m}\log\frac{1}{\delta}\,.
\end{equation*}
\label{thm:det_alg_otcs}
\end{theorem}
\begin{proof}
Due to \Cref{lem:truncation}, $d_{\psi}([\iprod{\theta_t}{X_t}]_b, \iprod{\theta^{\star}}{X_t}) \leq d_{\psi}(\iprod{\theta_t}{X_t}, \iprod{\theta^{\star}}{X_t})$. The claim then follows by combining \Cref{thm:algorithmic_general} and the discussion above.
\end{proof} 
If we apply any of these confidence set to sparse linear models, where $m=1$, and use the algorithm of
\citet{gerchinovitz2013sparsity} to generate the sequence $\theta_1, \dots, \theta_n$, then we obtain confidence sets 
of the form
\begin{equation}
\Theta_n = \left\{\theta \in \mathbb{R}^d\;:\; \sum_{t=1}^{n}\wh\ell_t(\theta) 
= \mathcal{O}\left(\max_tY_t^2s\log(\tfrac{dn}{s}) + \log\tfrac{1}{\delta}\right)\right\}\,.\label{eqn:seqsew_set}
\end{equation}
Here, the big-O notation hides large numerical constants and logarithmic factors of problem parameters such as 
$\norm{\thetastar}$ and $\sup_{t} \norm{X_t}$. The bound we will state in \Cref{thm:ewa_conversion} will spell 
all such dependencies out, and make strict improvements over the above guarantee. Since the above result uses a 
deterministic forecaster, the bound features a factor of $\max_tY_t^2$, which 
may generally contribute an extra factor of $b + M\log n$ to the confidence width (for $\sqrt{M}$-sub-Gaussian noise). 
We will address this limitation in the next section.

\paragraph{Comparison with state of the art.} In the case of linear models, our result is directly comparable with (and 
in fact inspired by) the result of Theorem~1 in \citet{abbasi2012online}, which shows that the confidence set defined 
in Equation~\eqref{eq:det_alg_set} is valid with the choice 
\begin{equation*}
\beta_n = \frac{1}{2} + 2\,\regret_{\theta^n}(\theta^{\star}) + 16\log\left(\frac{\sqrt{8} + \sqrt{1 + 2\,\regret_{\theta^n}(\theta^{\star})}}{\delta}\right)\,.
\end{equation*}
More generally, Theorem 1 by \citet{jun2017scalable} provides a comparable result for $M$-smooth and $m$-strongly 
convex GLMs, showing a confidence width of
\begin{equation*}
\beta_n = \frac{1}{2} + \frac{2}{m}\,\regret_{\theta^n}(\theta^{\star}) + \frac{4M}{m^2}\log\left(\frac{1}{\delta}\sqrt{4 + \frac{8}{m}\,\regret_{\theta^n}(\theta^{\star}) + \frac{64M^2}{4m^4\delta^2}}\right)\,.
\end{equation*}
In both cases, our confidence set has a much simpler expression, tighter numerical constants, and an improved 
dependence on $m$ in the case of GLMs. In both cases, the improvement is due to our use of \Cref{prop:shifted_lbd} 
instead of the self-normalized concentration inequality developed by \citet{abbasi2012online}.

\subsubsection{Algorithmic Conversions for the EWA Forecaster}
\label{sec:ewa_algorithmic}

We now provide an improved result that makes use of randomized forecasters, which will allow us to tighten the bounds 
above by removing the unnecessary $\max_t Y_t^2$ factor from the bound. Our confidence sets will take the same form as 
before, except that we will use pseudo-labels generated by the EWA forecaster with learning rate $\lambda = 1/2$, 
defined as follows:
\begin{equation}
\wh Y_t = \int \left[\inner{\theta}{X_t}\right]_b\mathrm{d}q_{t+1}(\theta)\,.\label{eqn:ewa_pseudo_label}
\end{equation}
The following theorem is an improvement of \Cref{thm:algorithmic_general}, which is specialized to the EWA forecaster.
\begin{theorem}
Let $q_1, q_2, \dots$ be the sequence of distributions played by the EWA forecaster with learning rate $\lambda=1/2$. 
Then,
\begin{equation*}
\PP{\exists n\,:\,\sum_{t=1}^n\log\int e^{d_{\psi}(\inner{\theta}{X_t}, \inner{\theta^\star}{X_t})}\dd q_{t+1}(\theta) 
\geq 
\frac{1}{2}\,\regret_{\qn,1/2}(\theta^\star) + \log\frac{1}{\delta}} \le \delta\,.
\end{equation*}
\label{thm:ewa_conversion}
\end{theorem}

\begin{proof}
As in the proof of \Cref{thm:algorithmic_general}, we use the shorthand $D_{t}(\theta) = \frac 12 \ell_t(\theta) + 
 \frac 12 \ell_t(\theta^\star) - \loss_{t}^{(1/2)}(\theta)$, and 
observe that $D_{t}(\theta) = d_{\psi}(\inner{\theta}{X_t}, \inner{\theta^\star}{X_t})$. From the definition of the EWA 
forecaster, $\frac{\dd q_{t+1}}{\dd q_t}(\theta) \propto e^{-\frac{1}{2}\ell_t(\theta)}$. Therefore, we have
\begin{equation*}
\int e^{D_{t}(\theta)}\dd q_{t+1}(\theta) = \frac{\int e^{-\ell_t^{(1/2)}(\theta)}\dd q_t(\theta)}{\int e^{-\frac{1}{2}\ell_t(\theta)}\dd q_t(\theta)}e^{\frac{1}{2}\ell_t(\theta^\star)}\,.
\end{equation*}
Taking logarithms and rearranging terms, we obtain
\begin{equation*}
\sum_{t=1}^n\log\int e^{D_t(\theta)}\dd q_{t+1}(\theta) = 
\frac{1}{2}\,\regret_{\qn,1/2}(\theta^\star) + \sum_{t=1}^{n}\ell_t(\theta^{\star}) - \sum_{t=1}^{n}\mathcal{L}_t^{(1/2)}(q_t)\,.
\end{equation*}
The statement then follows from applying \Cref{prop:shifted_lbd}.
\end{proof}

Using this result, we can prove the following improved version of \Cref{thm:det_alg_otcs}, which transforms the 
pseudo-labels generated by the EWA forecaster and the regret of the EWA forecaster against $\theta^{\star}$ into a 
confidence sequence.

\begin{theorem}
Suppose that $\thetastar$ satisfies $|\inner{\theta^{\star}}{X_t}| \leq b$ and that 
$\psi$ is $m$-strongly convex on $[-b,b]$. Let $q_1, q_2, \dots$ are the distributions played by the EWA forecaster 
with learning rate $\lambda = 1/2$. Then, for any $\delta\in(0,1)$, the set defined in 
Equation~\eqref{eq:det_alg_set} with $\wh Y_t$ defined in Equation \eqref{eqn:ewa_pseudo_label}
 satisfies $\PP{\exists n: 
\thetastar\not\in \Theta_n} \le \delta$ with the choice
\begin{equation*}
\beta_n = \frac{2}{m}\,\regretql{1/2}(\theta^{\star}) + \frac{4}{m}\log\frac{1}{\delta}\,.
\end{equation*}
\label{thm:ewa_alg_otcs}
\end{theorem}
\begin{proof}
Using Jensen's inequality, the strong convexity of $\psi$, \Cref{lem:truncation} and then Jensen's inequality again, we obtain
\begin{align*}
\frac 12 \pa{\iprod{\thetastar}{X_t} - \wh{Y}_t}^2 &\leq \frac{1}{2}\int([\inner{\theta}{X_t}]_b - \inner{\theta^{\star}}{X_t})^2\mathrm{d}q_{t+1}(\theta)\\
&\leq \frac{4}{m}\int d_{\psi}([\inner{\theta}{X_t}]_b, \inner{\theta^{\star}}{X_t})\mathrm{d}q_{t+1}(\theta)\\
&\leq \frac{4}{m}\int d_{\psi}(\inner{\theta}{X_t}, \inner{\theta^{\star}}{X_t})\mathrm{d}q_{t+1}(\theta)\\
&\leq \frac{4}{m}\log\int e^{d_{\psi}(\inner{\theta}{X_t}, \inner{\theta^{\star}}{X_t})}\mathrm{d}q_{t+1}(\theta)\,.
\end{align*}
The claim then follows by summing both sides over $t$ from $1$ to $n$, and then applying \Cref{thm:ewa_conversion}.
\end{proof}

Notably, the bound now depends on the regret of the EWA forecaster on the sequential probability assignment game with 
the \emph{logarithmic loss}, which allows us to remove the spurious $\max_t Y_t^2$ factor from the previous bound (that 
has resulted from using online algorithms with deterministic predictions).

\begin{theorem}
Suppose that $\psi$ is $M$-smooth on $\real$ and $m$-strongly convex on $[-b, b]$. In addition, suppose that $\theta^{\star}$ is $s$-sparse, $\|\theta^{\star}\|_2 \leq B$ and $\max_{t \geq 
1}\|X_t\|_{\infty} \leq L_{\infty}$. Let $q_1, q_2, \dots$ be the sequence of distributions played by the EWA forecaster 
with $\lambda = 1/2$ and the prior $q_1 = \sum_{S \subseteq [d]}\pi(S)q_S$, where $\rho(\theta) = \|\theta\|_2^2/(2B^2)$. Then the sequence of sets $(\Theta_t)_{t \geq 1}$, with $\Theta_n$ defined in 
Equation~\eqref{eq:det_alg_set} and $\wh Y_t$ defined in Equation \eqref{eqn:ewa_pseudo_label}, is a $1 - \delta$ 
confidence sequence with the choice 
\begin{equation*}
\beta_n = \frac{4s}{m}\log\frac{2ed\sqrt{1 + MB^2L_{\infty}^2n/2}}{s} + \frac{4}{m}\log\frac{2\sqrt{e}}{\delta}\,.
\end{equation*}
\label{cor:sparse_ewa_alg_otcs}
\end{theorem}

\begin{proof}
Let $S^{\star} = \mathrm{supp}(\theta^{\star})$ denote the support of $\theta^{\star}$. For any vector $x \in \mathbb{R}^d$, we define $x(S^{\star}) \in \mathbb{R}^{s}$ to be the $s$-dimensional subvector indexed by the elements of $S^{\star}$ (cf. \Cref{sec:res_breg_info}). Due to \Cref{thm:ewa_alg_otcs}, 
we only need to bound $\regretql{1/2}(\theta^{\star})$. We thus instantiate the regret bound of \Cref{lem:sparse} with the choice 
$\rho(\theta) = \frac{\|\theta\|_2^2}{2B^2}$, which yields
\begin{equation}
\regretql{1/2}(\theta^{\star}) \leq 2\left(\gamma_{n, 1/2}^{\rho, S^{\star}} + \frac{\|\theta^{\star}\|_2^2}{2B^2} + s\log\frac{2ed}{s} + \log 2\right)\,.\label{eqn:sparse_regret}
\end{equation}
By assumption, $\|\theta^{\star}\|_2^2/B^2 \leq 1$. Using \Cref{lemma:smoothinfo}, the restricted Bregman information gain is bounded as
\begin{equation*}
\gamma_{n, 1/2}^{\rho, S^{\star}} \leq \frac{1}{2}\log\det(\tfrac{MB^2}{2}\Lambda_n(S^{\star}) + \Id_s)\,,
\end{equation*}
where $\Lambda_n(S^{\star}) = \sum_{t=1}^{n}X_t(S^{\star})X_t(S^{\star})^{\top}$ and $\Id_s$ is the $s \times s$ identity matrix. Since, $\|X_t(S^{\star})\|_2 \leq \sqrt{s}L_{\infty}$, \Cref{lem:det_tr} tells us that
\begin{equation*}
\frac{1}{2}\log\det(\tfrac{MB^2}{2}\Lambda_n(S^{\star}) + \Id_s) \leq \frac{s}{2}\log(1 + MB^2L_{\infty}^2n/2)\,.
\end{equation*}
Combining this upper bound on the restricted Bregman information gain with the inequality in Equation \eqref{eqn:sparse_regret}, we see that
\begin{equation*}
\regretql{1/2}(\theta^{\star}) \leq 2s\log\frac{2ed\sqrt{1 + MB^2L_{\infty}^2n/2}}{s} + 2\log(2\sqrt{e})\,.
\end{equation*}
The statement then follows by substituting this regret bound into the expression for $\beta_n$ in \Cref{thm:ewa_alg_otcs}.
\end{proof}

\paragraph{Comparison with state of the art.} As promised, the bound in \Cref{cor:sparse_ewa_alg_otcs} removes the 
dependence on the maximum label that has appeared in previous bounds like Equation~\eqref{eqn:seqsew_set}. This is a 
clear improvement over the best previous results by \citet{abbasi2012online} and \citet{jun2017scalable} that both 
suffer from this factor. Additionally, the bound given above comes with small, explicit constants and a much simpler 
overall expression. These latter improvements come from using the regret bound of \Cref{lem:sparse} for EWA with a 
sparsity-inducing prior, instead of relying on the algorithm of \citet{gerchinovitz2013sparsity} whose regret 
bounds are much more complex to state. This complexity largely comes from their algorithm being completely 
parameter-free in the sense that it does not require any prior knowledge about the sparsity parameter $s$ or the norm 
of $\thetastar$. In our setting, knowledge of these parameters is necessary either way to evaluate the confidence 
width, which allows us to forgo this otherwise desirable parameter-free property.

\section{Discussion}\label{sec:conclusion}
We have introduced a framework that establishes a link between statistical inference and sequential 
prediction by providing a reduction scheme that allows constructing confidence sets for a broad class of statistical 
models. This work fits into a line of work on algorithmic statistics initiated by \citet{RS17}, which bridges
statistics and the theory of algorithms by establishing similar reductions between the problems in either of the two 
areas. We have demonstrated the effectiveness of this framework in yet another context, and applied our framework to 
recover and improve state-of-the-art results for statistical inference in generalized linear models, as well as provide 
completely new confidence sets for these models. Besides these concrete results, our framework enables further progress 
in the area by streamlining the process of proving new concentration inequalities, effectively reducing the otherwise 
complex task of proving concentration inequalities to the relatively simpler task of regret analysis of online 
algorithms. Accordingly, any new result in this actively researched area can be immediately turned into new 
concentration bounds. In this section, we close by discussing some further aspects our framework and results, as well 
as discuss questions that we leave open for future research.

\paragraph{Extensions.} For sake of concreteness and simplicity, we have focused on the class of generalized linear 
models with finite-dimensional parameters. Our framework can be extended in several straightforward ways. First, we 
mention that all of our results can be shown to also hold for ``sub-exponential families'' whose moment-generating 
functions satisfy the inequality $\mathbb{E}\left[\exp(\beta Y_t)|\mathcal{F}_{t-1}\right] \le 
\exp(\psi(\beta + \inner{\theta^\star}{X_t}) - \psi(\inner{\theta^\star}{X_t}))$ for some $\psi$. Indeed, it is easy to 
see that the conclusion of \Cref{prop:generic_lbd} continues to hold under these relaxed conditions, and thus all 
subsequent results can be generalized in this way. In particular, all results proved for linear models can be also 
shown to hold under sub-Gaussian noise. Second, we mention that our results can be directly extended to 
infinite-dimensional models where the linear function $\iprod{\thetastar}{\cdot}$ is replaced by some $f^\star$ that 
belongs to a reproducing kernel Hilbert space (see, e.g., \citealp[Chapter~3]{abbasi2012thesis}, 
\citealp{emmenegger2023likelihood} or \citealp{flynn2024tighter}). We opted to not include results for this case to 
keep the paper less technical and easier to read, and leave working out the (potentially nontrivial) details for future 
work.

\paragraph{The tightness of our confidence sets.} For all cases we have studied in Section~\ref{sec:cases}, our 
results either recover the best known guarantees or make improvements over them. In some important special cases, these 
results also match the best achievable bounds: for linear models, the bounds under adaptively chosen 
covariates (Section~\ref{sec:cases_adaptive}) match the lower bound of \citet{lattimore2023lower}, and the bounds for 
obliviously chosen covariates (Section~\ref{sec:cases_oblivious}) can only be improved in terms of constant 
factors (see, e.g., Example 15.14 in \citealp{wainwright2019high}). This suggests that our reduction technique does not 
introduce any gaps between the best 
achievable regret and the best achievable concentration inequalities (which is consistent with the general 
equivalence results between the two types of bounds proved by \citealp{RS17}). As for the tightness of the 
confidence sets and regret bounds we provide in this paper, we believe that several improvements should be possible. In 
particular, under certain regularity conditions (such as boundedness of the covariates and the true parameter 
$\thetastar$) the minimax regret for sequential probability assignment with GLMs is known to scale 
asymptotically as $\frac{d\log n}{2}$ \citep{grunwald2007minimum}. Finite-sample guarantees that match these rates are 
quite rare in 
the literature. One notable exception beyond linear models is the work of \citet{jacquet2022precise} that provides a 
finite-sample upper bound on the regret of the NML forecaster for logistic regression with bounded covariates, which 
matches the minimax rate up to some explicitly given additional terms that vanish as $n$ grows large. In comparison to 
these bounds, the ones we provide are loose by a factor of the strong convexity constant $1/m$, which is often very 
large in GLMs of practical interest (e.g., as noted earlier, it can be exponentially large in $d$ for the important case 
of logistic regression). It remains unclear if further improvements are possible at the level of generality that we have 
considered in this work, and we highlight this question as the most exciting one for future research.

\paragraph{Connections with mean estimation.} It is insightful to instantiate our confidence sets in the special case 
of (one-dimensional) mean estimation, recovered by setting $X_t = 1$ for all $t$. In case the noise is sub-Gaussian, 
instantiating our results for obliviously chosen covariates (\Cref{thm:transductive}) correctly recovers Hoeffding's 
inequality without any additional $\log n$ factors (up to a small difference in constants). The case 
of observations that are almost 
surely bounded in $[0,1]$ can be handled by noticing that such random variables are ``sub-Bernoulli'' in the sense 
defined in the paragraph above, and thus a reduction to logistic regression ($\psi:z\mapsto\log(1+e^z)$) can be used. 
In this case, the shape of our confidence bounds matches the shape implied by \citet{orabona2024tight}, which has been 
recently shown to be optimal in an appropriate sense \citep{clerico2024optimality}. Notably however, the width of our 
confidence interval in this case is only guaranteed to be meaningfully bounded if the range of the parameter $p = 
\EE{Y_t}$ is constrained to lie strictly within the limits of the unit interval. Concretely, if it is known 
that $p\in [p_0,(1-p_0)]$ holds for some known $p_0 > 0$, then the loss function $\loss_t$ can be shown to be 
strongly convex with parameter $m = \frac{1}{p_0(1-p_0)}$, and applying \Cref{thm:transductive} gives a confidence 
width of order $\log \pa{\frac{1}{p_0(1-p_0)}}$. This recovers the confidence interval of \citet{orabona2024tight} 
whenever $p_0 = \Omega\pa{\frac{1}{n}}$, which we consider to be a reasonable regime. More generally, our bounds 
recover the results of \citet{bregman2023} in the case of one-dimensional exponential families (as can be seen by 
directly combining our \Cref{thm:main-reduction} with \Cref{lemma:regrewa}).

\paragraph{NML and minimax optimality.} As we have noted in Section~\ref{sec:NML}, the Normalized Maximum Likelihood 
forecaster is minimax optimal for the problem of sequential probability assignment (up to a regularization 
term in the version we have described). Unfortunately however, directly bounding the regret of NML (as 
expressed by the Shtarkov sum) is very challenging even for relatively simple models like GLMs. Beyond the simple 
linear case (where NML is essentially equivalent to EWA), the only bounds we are aware of only hold asymptotically or 
for mean estimation without covariates \citep{Rissanen1996,TakeuchiB97,Szpankowski1998,cesa2006prediction,grunwald2007minimum,SuzukiYamanishi2018,jacquet2022precise}. We also remark that the equalizer property of 
NML-style forecasters is not always a desirable feature in the context of our study. Indeed, notice that several of our 
guarantees are stated in terms of the realized sequence of covariates, and the worst-case upper bounds we provide in 
terms of problem parameters like $d$ or $n$ are often loose (see, e.g., the bound of Equation~\eqref{eq:regret_rank} 
and the discussion around it). Thus, equalizer strategies that take into account the impact of predictions on the 
sequence of covariates (such as the cNML strategy of \citealp{liu2024sequential}) may result in unnecessarily large 
regret when the realized sequence turns out to be benign. Studying the best achievable regret for sequential 
probability assignment remains an actively researched area, and we remain hopeful that new developments (such as 
\citealp{qian2024refined,jia2025minimax}) will enable proving tighter finite-sample bounds and a more fine-grained 
understanding of the best possible rates for special cases like GLMs and beyond.

\paragraph{Uncertainty quantification beyond GLMs.} The reduction framework we propose in this paper is clearly 
applicable for more general models, and we chose to focus on GLMs for the sake of concreteness. Indeed, it is easy to 
see that our arguments continue to hold regardless of the form of the functional dependence between the parameters of 
the model and the likelihood, and that one can derive convex confidence sets using our method whenever the 
log-likelihood is convex (as is the case, for example, for additive models and generalized additive models, 
cf.~\citealp{hastie1986generalized}). It is much more interesting to study the possibility of extending our framework 
towards 
non-convex models such as deep neural networks, and develop a rigorous methodology for statistical inference 
for modern machine learning systems. We believe that our methodology could prove to be a solid basis for addressing 
``uncertainty quantification'' tasks for machine learning, such as providing confidence sets for local minimizers of 
the population loss, or valid prediction intervals for individual data points. Obviously, our methodology is already 
applicable as it is for non-convex log-likelihoods, but its usefulness is limited because in such settings it leads to 
non-convex confidence sets and prediction intervals. A promising direction towards removing this limitation is to build 
on recently popular techniques for local linearization of nonlinear models (such as Laplace approximations, 
cf.~\citealp{Mac92,KIAK19,DKIEBH21,IKB21,CPFHB24}), and building confidence sets on the resulting linear 
model. We are optimistic that our method that directly addresses statistical questions via reductions to algorithmic 
questions can prove to be powerful in this highly challenging setting, especially given the abundance of algorithmic 
results and the relative scarcity of statistical guarantees in this context.

\begin{appendix}\label{appn} 

\section{Restricted Bregman Information Gain}
\label{sec:res_breg_info}

In this section, we justify why $\gamma_{n,\lambda}^{\rho, S}$ is called the restricted Bregman information gain. For any subset $S \subseteq [d]$ and any vector $x \in \mathbb{R}^d$, we define $x(S) \in \mathbb{R}^{|S|}$ to be the subvector indexed by elements of $S$. This can be written as $x(S) = \Pi_Sx$, where $\Pi_S \in \mathbb{R}^{|S|\times d}$ is the matrix with rows $(e_i)_{i \in S}$, and $e_1, \dots, e_d$ are the standard basis vectors of $\mathbb{R}^d$. For each subset $S \subseteq [d]$, we will define a restricted version of $Z_{n,\lambda}^{\rho}$ which has domain $\mathbb{R}^{|S|}$. For any vector $v \in \mathbb{R}^{|S|}$, we define the restricted loss $\ell_t^S$ as
\begin{equation*}
\ell_t^S(v) = -\inner{v}{X_t(S)}Y_t + \psi(\inner{v}{X_t(S)}) -\log h(Y_t)\,.
\end{equation*}
In addition, we define a collection of convex functions $\rho^S: \mathbb{R}^{|S|} \to \mathbb{R}$ such that for all $S \subseteq [d]$ and $\theta \in \Theta_S$, $\rho(\theta) = \rho^{S}(\theta(S))$. Note that if $\rho(\theta) = \|\theta\|_p^p$, then we can define $\rho^S$ by $\rho^S(v) = \|v\|_p^p$. We can now define $Z_{n,\lambda}^{\rho,S}: \mathbb{R}^{|S|} \to \mathbb{R}$ to be the function
\begin{equation*}
Z_{n,\lambda}^{\rho, S}(v) = \sum_{t=1}^{n}\lambda\ell_t^S(v) + \rho^S(v)\,.
\end{equation*}
We notice that for all $S \subseteq [d]$ and $\theta \in \Theta_S$, $Z_{n,\lambda}^{\rho, S}$ satisfies $Z_{n,\lambda}^{\rho}(\theta) = Z_{n,\lambda}^{\rho, S}(\theta(S))$. If we let $\wh\theta_{n,\lambda,S} = \argmin_{\theta \in \Theta_S}\{Z_{n,\lambda}^{\rho}(\theta)\}$, then $\wh\theta_{n,\lambda,S}(S) = \argmin_{v \in \mathbb{R}^{|S|}}\{Z_{n,\lambda}^{\rho, S}(v)\}$. This means that, for any $S \subseteq [d]$ and $\theta \in \Theta_S$,
\begin{equation*}
Z_{n,\lambda}^{\rho}(\theta) - Z_{n,\lambda}^{\rho}(\wh\theta_{n,\lambda,S}) = Z_{n,\lambda}^{\rho,S}(\theta(S)) - Z_{n,\lambda}^{\rho,S}(\wh\theta_{n,\lambda,S}(S)) = \mathcal{B}_{Z_{n,\lambda}^{\rho,S}}(\theta(S), \wh\theta_{n,\lambda,S}(S))\,.
\end{equation*}
Therefore, for any subset $S \subseteq [d]$, the definition of the restricted Bregman information gain in Equation \eqref{eqn:res_breg_info} is equivalent to
\begin{equation*}
\gamma_{n,\lambda}^{\rho, S} = -\log\left(\frac{\int_{\Theta_S}\exp(-\mathcal{B}_{Z_{n,\lambda}^{\rho,S}}(\theta(S), \wh\theta_{n,\lambda,S}(S)))\mathrm{d}\theta}{\int_{\Theta_S}\exp(-\rho^S(\theta(S))\mathrm{d}\theta}\right)\,.
\end{equation*}

\section{Technical Lemmas}

\begin{lemma}
Let $q_1, \dots, q_n$ be the sequence of distributions played by the EWA forecaster with learning rate $\lambda > 0$. Then
\begin{equation*}
\regretql{\lambda}(\bar{\theta}) = -\frac{1}{\lambda}\log\int e^{-\lambda\sum_{t=1}^{n}\big(\ell_t(\theta) - 
\ell_t(\bar{\theta})\big)}\dd q_1(\theta)\,.
\end{equation*}
\label{lem:telescope}
\end{lemma}

\begin{proof}
By definition of the EWA forecaster,
\begin{equation*}
\frac{\mathrm{d}q_t}{\mathrm{d}q_1}(\theta) = \frac{e^{-\lambda\sum_{k=1}^{t-1}\ell_k(\theta)}}{\int e^{-\lambda\sum_{k=1}^{t-1}\ell_k(\theta)}\mathrm{d}q_1(\theta)}\,.
\end{equation*}
Therefore, the total loss of the EWA forecaster is
\begin{align*}
\sum_{t=1}^{n}\mathcal{L}_{t, \lambda}(q_t) &= \sum_{t=1}^{n}-\frac{1}{\lambda}\log\int e^{-\lambda\ell_t(\theta)}\mathrm{d}q_t(\theta)\\
&= \sum_{t=1}^{n}-\frac{1}{\lambda}\log\left(\frac{\int e^{-\lambda\sum_{k=1}^{t}\ell_k(\theta)}\mathrm{d}q_1(\theta)}{\int e^{-\lambda\sum_{k=1}^{t-1}\ell_k(\theta)}\mathrm{d}q_1(\theta)}\right)\\
&= -\frac{1}{\lambda}\log\int e^{-\lambda\sum_{t=1}^{n}\ell_t(\theta)}\mathrm{d}q_1(\theta)\,.
\end{align*}
The statement follows by subtracting the total loss of $\bar{\theta}$ from both sides.
\end{proof}

\begin{lemma}
Let $x_1, \dots, x_n \in \mathbb{R}^d$ be any sequence satisfying $\max_{t \in [n]}\|x_t\|_2 \leq L$ and set $\Lambda_n = \sum_{t=1}^{n}x_tx_t^{\top}$. Then for every $\alpha > 0$,
\begin{equation*}
\log\det(\alpha\Lambda_n + \Id) \leq \mathrm{rank}(\Lambda_n)\log\left(1 + \frac{\alpha n L^2}{\mathrm{rank}(\Lambda_n)}\right)\,.
\end{equation*}
\label{lem:det_tr}
\end{lemma}
\begin{proof}
Let $k = \mathrm{rank}(\Lambda_n)$ and let $\lambda_1, \dots, \lambda_d$ be the eigenvalues of $\alpha\Lambda_n + \Id$, arranged in descending order. Since $\Lambda_n$ has rank $k$, $\lambda_{k+1} = \lambda_{k+2} = \cdots = \lambda_d = 1$. This means that $\det(\alpha\Lambda_n + \Id) = \prod_{i=1}^{k}\lambda_k$. Similarly, $\mathrm{tr}(\alpha\Lambda_n + \Id) = \sum_{i=1}^{k}\lambda_k + (k-d)$. Using the inequality of arithmetic and geometric means, we have
\begin{equation*}
\det(\alpha\Lambda_n + \Id) = \prod_{i=1}^{k}\lambda_k \leq \left(\frac{\sum_{i=1}^{k}\lambda_i}{k}\right)^{k} = \left(\frac{\mathrm{tr}(\alpha\Lambda_n + \Id) + (k - d)}{k}\right)^{k}\,.
\end{equation*}
Finally, we notice that
\begin{equation*}
\mathrm{tr}(\alpha\Lambda_n + \Id) = \alpha\sum_{t=1}^{n}\|x_t\|_2^2 + d \leq \alpha nL^2 + d\,.
\end{equation*}
Combining this with the previous inequality and taking the logarithm of both sides gives the statement of the lemma.
\end{proof}

\begin{lemma}
For any $z^{\prime} \in \mathbb{R}$, the map $z \mapsto d_{\psi}(z, z^{\prime})$ is non-decreasing on the interval $[z^{\prime}, \infty)$ and non-increasing on the interval $(-\infty, z^{\prime}]$.
\label{lem:quasi}
\end{lemma}

\begin{proof}
Let $f(z) = d_{\psi}(z, z^{\prime})$. From the definition of $d_{\psi}$, we have
\begin{equation*}
f^{\prime}(z) = \frac{1}{2}\psi^{\prime}(z) - \frac{1}{2}\psi^{\prime}(z/2 + z^{\prime}/2)\,.
\end{equation*}

Since $\psi$ is convex and differentiable, $\psi^{\prime}$ is non-decreasing. We notice that for any $z \geq z^{\prime}$, $z \geq z/2 + z^{\prime}/2$, which means that $f^{\prime}(z) \geq 0$. The same argument shows that for any $z \leq z^{\prime}$, $f^{\prime}(z) \leq 0$.
\end{proof}

\begin{lemma}
For any $b > 0$, any $z^{\prime} \in [-b, b]$ and all $z \in \real$, $d_{\psi}([z]_b, z^{\prime}) \leq d_{\psi}(z, z^{\prime})$.
\label{lem:truncation}
\end{lemma}

\begin{proof}
Suppose $z \geq z^{\prime} \geq -b$. If $z \leq b$, then $d_{\psi}([z]_b, z^{\prime}) = d_{\psi}(z, z^{\prime})$. Otherwise, if $z > b$, then $z > [z]_b \geq z^{\prime}$. Thus, $d_{\psi}([z]_b, z^{\prime}) \leq d_{\psi}(z, z^{\prime})$ follows from \Cref{lem:quasi}. The same argument works when $z \leq z^{\prime}$.
\end{proof}

\end{appendix}


\begin{acks}[Acknowledgments]
The authors would like to thank the following colleagues for insightful discussions throughout the preparation of this 
work: Yoav Freund, Kwang-Sung Jun, Johannes Kirschner, Junghyun Lee, Nishant Mehta, Francesco Orabona, Sasha Rakhlin, 
Aaditya Ramdas, Csaba Szepesv\'ari.
\end{acks}

\begin{funding}
This project has received funding from the European Research Council (ERC), under the European
Union’s Horizon 2020 research and innovation programme (Grant agreement No.~950180).
\end{funding}

\bibliographystyle{imsart-nameyear}
\bibliography{main,ngbib}

\end{document}